\documentclass[12pt,a4paper,leqno,twoside]{amsart}
\usepackage[a4paper,left=3cm,top=4cm,bottom=4cm,right=3cm]{geometry}
\usepackage[T1]{fontenc}
\usepackage{amssymb,bbm}
\usepackage{amsmath,amsthm,dsfont}
\usepackage{amsfonts,mathrsfs,wasysym}
\usepackage{latexsym}
\usepackage[cp1250]{inputenc}
\usepackage{graphicx}
\usepackage{indentfirst}
\usepackage{enumerate}
\usepackage{enumitem}
\usepackage{xparse}
\usepackage{etoolbox}
\usepackage{mathptmx}
\usepackage{fancyhdr}
\usepackage[dvipsnames]{xcolor}
\usepackage{marvosym}
\usepackage{pifont}
\usepackage{cite}

\usepackage{tikz}
\usepackage{scalerel}

\usepackage{graphics}
\usepackage{pstricks}
\usepackage{mdframed}

\usepackage{hyperref}
\hypersetup{colorlinks,linkcolor={blue},citecolor={red}}

\newcommand{\zn}{\scaleto{0}{1ex}\scaleto{n}{0.84ex}}
\newcommand{\mzn}{\scaleto{0}{0,8ex}\scaleto{n}{0.7ex}}

\patchcmd{\abstract}{\scshape\abstractname}{\textbf{\abstractname}}{}{}

\makeatletter
\DeclareMathAlphabet{\mathcal}{OMS}{cmsy}{m}{n}

\DeclareSymbolFont{operators}{OT1}{ztmcm}{m}{n}
\DeclareSymbolFont{letters}{OML}{ztmcm}{m}{it}
\DeclareSymbolFont{symbols}{OMS}{ztmcm}{m}{n}
\DeclareSymbolFont{largesymbols}{OMX}{ztmcm}{m}{n}
\DeclareSymbolFont{bold}{OT1}{ptm}{bx}{n}
\DeclareSymbolFont{italic}{OT1}{ptm}{m}{it}
\@ifundefined{mathbf}{}{\DeclareMathAlphabet{\mathbf}{OT1}{ptm}{bx}{n}}
\@ifundefined{mathit}{}{\DeclareMathAlphabet{\mathit}{OT1}{ptm}{m}{it}}
\DeclareMathSymbol{\omicron}{0}{operators}{`\o}

\DeclareMathAlphabet{\mathpzc}{OT1}{pzc}{m}{it}
\DeclareSymbolFont{operators}{OT1}{txr}{m}{n}
\SetSymbolFont{operators}{bold}{OT1}{txr}{bx}{n}
\def\operator@font{\mathgroup\symoperators}
\DeclareSymbolFont{italic}{OT1}{txr}{m}{it}
\SetSymbolFont{italic}{bold}{OT1}{txr}{bx}{it}
\DeclareSymbolFontAlphabet{\mathrm}{operators}
\DeclareMathAlphabet{\mathbf}{OT1}{txr}{bx}{n}
\DeclareMathAlphabet{\mathit}{OT1}{txr}{m}{it}
\SetMathAlphabet{\mathit}{bold}{OT1}{txr}{bx}{it}
\DeclareSymbolFont{letters}{OML}{txmi}{m}{it}
\SetSymbolFont{letters}{bold}{OML}{txmi}{bx}{it}
\DeclareFontSubstitution{OML}{txmi}{m}{it}
\DeclareSymbolFont{lettersA}{U}{txmia}{m}{it}
\SetSymbolFont{lettersA}{bold}{U}{txmia}{bx}{it}
\DeclareFontSubstitution{U}{txmia}{m}{it}
\DeclareSymbolFontAlphabet{\mathfrak}{lettersA}
\DeclareSymbolFont{symbols}{OMS}{txsy}{m}{n}
\SetSymbolFont{symbols}{bold}{OMS}{txsy}{bx}{n}
\DeclareFontSubstitution{OMS}{txsy}{m}{n}
\makeatother

\newcommand{\adm}[2]{%
    \mathrel{%
        \text{%
            \begin{tikzpicture}[baseline=0ex]
                \draw[line width=0.9pt] (0,#1) -- (0,#2); 
            \end{tikzpicture}%
        }%
    }%
}

\makeatletter
\renewcommand\abstractname{\scshape\bfseries Abstract}

\renewenvironment{proof}[1][\proofname]{\par \pushQED{\qed} \normalfont
  \topsep6\p@\@plus6\p@ \trivlist \itemindent\z@
  \item[\hskip\labelsep\bfseries
    #1\@addpunct{.}]\ignorespaces
}{
  \popQED\endtrivlist\@endpefalse
}

\makeatother

\makeatletter
    \@addtoreset{equation}{section} 
    \renewcommand{\theequation}{{\thesection}.\@arabic\c@equation} 

\makeatother

\makeatletter

\def\section{\@ifstar\unnumberedsection\numberedsection}
\def\numberedsection{\@ifnextchar[
  \numberedsectionwithtwoarguments\numberedsectionwithoneargument}
\def\unnumberedsection{\@ifnextchar[
  \unnumberedsectionwithtwoarguments\unnumberedsectionwithoneargument}
\def\numberedsectionwithoneargument#1{\numberedsectionwithtwoarguments[#1]{#1}}
\def\unnumberedsectionwithoneargument#1{\unnumberedsectionwithtwoarguments[#1]{#1}}
\def\numberedsectionwithtwoarguments[#1]#2{%
  \ifhmode\par\fi
  \removelastskip
  \vskip 4ex\goodbreak
  \refstepcounter{section}%
  \noindent
  \begingroup
  \leavevmode\centering\scshape\bfseries
  \thesection.
  #2
  \par
  \endgroup
  \vskip 1ex\nobreak
  \addcontentsline{toc}{section}{%
    \protect\numberline{\thesection}%
    #1}%
  }
\def\unnumberedsectionwithtwoarguments[#1]#2{%
  \ifhmode\par\fi
  \removelastskip
  \vskip 2ex\goodbreak
  \noindent
  \begingroup
  \leavevmode\centering\scshape\bfseries
  \leavevmode\centering\scshape\bfseries
  #2
  \par
  \endgroup
  \vskip 1ex\nobreak
  \addcontentsline{toc}{section}{%
    #1}%
}
\makeatother

\makeatletter
\def\@seccntformat#1{\csname mythe#1\endcsname}

\let\latex@subsection\subsection
\def\subsection{\@ifstar{\refstepcounter{subsection}\latex@subsection*}{\latex@subsection}}
\def\@makechapterhead#1{%
  \vspace*{40\p@}%
  {\parindent \z@ \raggedright \normalfont
    \interlinepenalty\@M
    \Huge \bfseries #1\par \nobreak
    \vskip 40\p@
  }}
\let\latex@l@chapter\l@chapter
\def\l@chapter#1#2{\begingroup\let\numberline\@gobble\latex@l@chapter{#1}{#2}\endgroup}
\makeatother

\theoremstyle{plain}
\newtheorem{Th}{Theorem}[section]
\newtheorem{Prop}[Th]{Proposition}
\newtheorem{Lem}[Th]{Lemma}
\newtheorem{Cor}[Th]{Corollary}
\theoremstyle{definition}
\newtheorem{Rem}[Th]{Remark}
\newtheorem{Ex}[Th]{Example}

\newtheorem{Def}[Th]{Definition}

\def\bf{\textbf}
\def\it{\textit}
\def\tn{\textnormal}
\def\leq{\leqslant}
\def\geq{\geqslant}
\def\R{{\mathds R}}
\def\N{{\mathds N}}
\def\B{\textit{I\!B}}

\def\D{{\mathrm{dom}}\,}
\def\E{{\mathrm{epi}}\,}
\def\G{{\mathrm{gph}}\,}

\def\dist{\mathrm{dist}}

\begin{document}
\vspace*{-4mm}
\title{Stable representations of Hamilton-Jacobi-Bellman equations with infinite horizon}

\author{\vspace*{-0.2cm}{Arkadiusz Misztela$^{\tn{\textdagger}}$ and S\l{}awomir Plaskacz $^{\tn{\textdaggerdbl}}$}\vspace*{-0.2cm}}
\thanks{\textdagger\,Institute of Mathematics, University of Szczecin, Wielkopolska 15, 70-451 Szczecin, Poland; e-mail: arkadiusz.misztela@usz.edu.pl ({\color{Green}Corresponding author})}
\thanks{\textdaggerdbl\,Nicholas Copernicus University, Faculty of Economic Sciences and Management, Gagarina 13A,\linebreak 87-100 Toruń, Poland; e-mail: plaskacz@mat.umk.pl}

\begin{abstract}
In this paper, for the Hamilton-Jacobi-Bellman equation with an infinite horizon and state constraints, we construct a suitably regular representation. This allows\linebreak us to reduce the problem of existence and uniqueness of solutions to the Frankowska and Basco theorem from \cite{B-F-2019}. Furthermore, we demonstrate that our representations are stable. The obtained results are illustrated with examples.\\
\vspace{0mm}

\hspace{-1cm}
\hspace*{-0.55cm} state constraints, stability of solutions, representation of Hamiltonians.

\vspace{3mm}\hspace{-1cm}
\end{abstract}

\maketitle

\pagestyle{myheadings}  \markboth{\small{\scshape Arkadiusz Misztela and Sławomir Plaskacz}
}{\small{\scshape Hamilton-Jacobi-Bellman Equations}}

\thispagestyle{empty}

\vspace{-0.8cm}


\section{Introduction}

\noindent Value functions for optimal control problems given by a dynamics $f(t,x,u)$, a running cost function $l(t,x,u)$, and a set of controls $U(t)$ are known to be  a weak solution to a corresponding Hamilton-Jacobi-Bellman equation
\begin{equation}\label{eqhjb_0}
-W_t+H(t,x,-W_x)=0,
\end{equation}
where the Hamiltonian $H(t,x,p)$ is given by
\begin{equation}\label{def-rh0}
H(t,x,p)=\sup\nolimits_{u\,\in\,U(t)}\,\{\,\langle\, p,f(t,x,u)\,\rangle-l(t,x,u)\,\}.
\end{equation}

\noindent The uniqueness of viscosity solutions to \eqref{eqhjb_0} has been obtained by Crandall, Evans, and Lions in  \cite{Lions1, Lions2}. The method of comparison between viscosity sub- and supersolutions used in that pioneering works can be used for a wide class of Hamiltonians not necessarily\linebreak related to a control system, in particular for Hamiltonians $H(t,x,p)$  that are not convex\linebreak with respect to the last variable.
H. Frankowska in \cite{HF} obtained the existence and\linebreak uniqueness of weak solutions of the Hamilton-Jacobi-Bellman equations corresponding to a Mayer problem by showing that the value function is the unique weak solution. This\linebreak method has been adopted in  \cite{B-F-2019-l,B-F-2019,F-P-200}. In the cited papers, the assumptions are not\linebreak formulated directly as the regularity and boundedness properties of the Hamiltonian.  The assumptions are expressed as some properties of the dynamics $f(t,x,u)$ and the running\linebreak cost $l(t,x,u)$. Hamilton-Jacobi equations appear in problems where the Hamiltonian\linebreak $H(t,x,p)$ is not related to a control problem (comp. \cite{Cieslak}), but nevertheless is convex with respect to the last variable. This type of problem and pure cognitive curiosity motivated\linebreak  authors to construct a representation $(U,f,l)$ of a given Hamiltonian $H(t,x,p)$. The\linebreak triple $(U,f,l)$ is a representation of the Hamiltonian $H(t,x,p)$ if equality \eqref{def-rh0} holds true.\linebreak The problem of finding an appropriate representation of the given Hamiltonian has been\linebreak considered in \cite{FR,AM2,AM3,F-S,AM0,AM,AM1,AMaX2,B-F-2020,B-F-2022}.

In the paper, we consider the state constrained Hamilton-Jacobi-Bellman (H-J-B)\linebreak equation with infinite horizon terminal condition, described as follows:
\begin{equation}\label{eqhjb}
\left\{\begin{array}{l}
-W_{t}+ H(t,x,-W_{x})=0\;\;\; \tn{in}\;\;\; (0,\infty)\times A, \\[1mm]
\lim\nolimits_{\,t\to\infty}\sup\nolimits_{\,x\,\in\,\D W(t,\,\cdot\,)}|W(t,x)|=0,
\end{array}\right.
\end{equation}
in which $H:[0,\infty)\times \R^{\scriptscriptstyle N}\times\R^{\scriptscriptstyle N}\to \R$ is a given Hamiltonian that is convex with respect to the last variable, $A\subset \R^{\scriptscriptstyle M}$ is a given closed set and $W:[0,\infty)\times A\to\R\cup\{+\infty\}$ is an unknown weak solution.
The problem \eqref{eqhjb} is related to an infinite horizon optimal control problems with state constraints $A$ given by a triple $(U,f,l)$, where $U:[0,\infty)\to \R^{\scriptscriptstyle M}$ is a set-valued map with nonempty values of controls, $f:[0,\infty)\times\R^{\scriptscriptstyle N}\times\R^{\scriptscriptstyle M}\to\R^{\scriptscriptstyle N}$ is a dynamics of a control system and $l:[0,\infty)\times\R^{\scriptscriptstyle N}\times\R^{\scriptscriptstyle M}\to\R$ is a running cost function.
The value function $\mathcal{V}:[0,\infty)\times A\to\R\cup\{+\infty\}$ of the optimal control problem is defined by
\begin{equation}\label{def-vf2}
\mathcal{V}(t_0,x_0)=\inf_{(x,u)(\cdot)\,\in\, S_{\!\!f\,}(t_0,x_0)}\int_{t_0}^{\infty}l(t,x(t),u(t))\,dt,
\end{equation}
where $S_{\!\!f\,}(t_0,x_0)$ denotes the set of all trajectory-control pairs of the control system
\begin{equation}\label{def-vf2-svs}
\left\{\begin{array}{l}
\dot{x}(t)=f(t,x(t),u(t)),\;\;u(t)\in U(t),\;\;\tn{a.e.}
\;\;t\in[t_0,\infty), \\[1mm]
x(t_0)=x_0,\quad x([t_0,\infty))\subset A.
\end{array}\right.
\end{equation}
 V. Basco and H. Frankowska in \cite{B-F-2019} give a list of assumptions onto the control system $(U,f,l)$ and obtain that the value function $\mathcal{V}(t,x)$ is the unique weak solution to \eqref{eqhjb}.

In this paper, attention focuses on hypotheses on the Hamiltonian $H(t,x,p)$ that\linebreak allows to construct a control system $(U,f,l)$ being its representation.
Using the methods of constructing an epigraphical representation introduced by A. Misztela in \cite{AM,AM1} we\linebreak reformulate the results of Basco-Frankowska. We provide $\tn{(OPC)}_{\!H}\!$ and a list (denoted $\!\tn{(h)}_{\!H}''$)\linebreak  of properties of the Hamiltonian $H(t,x,p)$ and its Fenchel conjugate $ H^{\ast}(t,x,v)$ that are\linebreak sufficient to construct a representation  $(\B,f,l)$ of the Hamiltonian $H(t,x,p)$ satisfying the\linebreak assumption of Theorem 3.3 in \cite{B-F-2019} ($\B$ denotes the closed unit ball in $\R^{\scriptscriptstyle N+1}$). For the reader's convenience, we adopt the notations from \cite{B-F-2019}, where the list of properties of the triple\linebreak $(U,f,l)$, which are the assumptions of Theorem 3.3, is denoted by $\tn{(h)}''$ and (OPC)\linebreak (we recall it in Section \ref{sec-rch}). Moreover, we show that the Hamiltonian obtained by \eqref{def-rh0}\linebreak from a triple $(U,f,l)$ satisfying the assumptions of Theorem 3.3 in \cite{B-F-2019} meets the full list $\tn{(h)}_H''$ and $\tn{(OPC)}_{H}$. Using our  representation of the Hamiltonian, we obtain the existence and uniqueness results for the problem \eqref{eqhjb} when all  assumptions are formulated as the properties of the Hamiltonian. The existence result for the problem \eqref{eqhjb} additionally requires the condition $\tn{(B)}_H$, which is a weaker version of the condition (B) from  \cite{B-F-2019}.

To obtain stability for  \eqref{eqhjb} we restrict the class of Hamiltonians and we additionally assume that the dynamics given by the domain of the Fenchel conjugate $H^*(t,x,\cdot)$ is forward viable and backward invariant to the constraints set $A$. This restricted class contains in particular  Hamiltonians related to a control triple $(U,f,e^{-\gamma t}l)$, where $f,l$ are bounded.

The outline of this paper is as follows. In Section \ref{sec-prel} we provide some notations. In Section \ref{sec-rch} we obtain a representation of Hamiltonian that in the next section allows to obtain the existence and uniqueness of weak solutions to \eqref{eqhjb}. In the last section we discuss the stability of weak solutions to \eqref{eqhjb} when Hamiltonians $H_n$ converge to $H$.

\section{Preliminaries}\label{sec-prel}

\noindent We denote by $|\cdot|$ and $\langle \cdot,\cdot\rangle$ the Euclidean norm and scalar product in $\R^{\scriptscriptstyle N}$, respectively. For a nonempty subset $C\subset\R^{\scriptscriptstyle N}$ we denote the interior of $C$ by $\mathrm{int}\,C$, the closure $\mathrm{cl}\,C$, the boundary of $C$ by $\mathrm{bd}\,C$, the convex hull of $C$ by $\mathrm{con}\,C$, the norm of $C$ by $\|C\|=\sup_{x\in C}|x|$, and the distance from  $x\in\R^{\scriptscriptstyle N}$ to $C$ by $\dist(x,C)=\inf_{y\in C}|\,x-y\,|$. Let $\B(x,r)$ stand for the closed ball in $\R^{\scriptscriptstyle N}$ with radius $r\geq 0$ centered at $x\in\R^{\scriptscriptstyle N}$ and $\B:=\B(0,1)$, $\mathds{S}:=\mathrm{bd}\,\B$. The extended Hausdorff distance between nonempty subsets $C$, $D$ of $\R^{\scriptscriptstyle N}$ is defined by
\begin{equation*}
\textit{d\!l}_\mathcal{H}(C,D)= \max\big\{\,\sup\nolimits_{x\in C}\dist(x,D),\;\sup\nolimits_{x\in D}\dist(x,C)\,\big\}\in\R\cup\{+\infty\}.
\end{equation*}

Let $I$ and $J$ be two intervals in $\R$. We denote by $L^1(I;J)$ the set
of all $J$-valued Lebesgue\linebreak integrable functions on $I$. We say that $\varphi\in L^1_{\mathrm{loc}}(I;J)$ if $\varphi\in L^1_{\mathrm{loc}}(\textit{I\!I};J)$ for any compact sub-interval $\textit{I\!I}\subset I$. In what follows $\mu$ stands for the Lebesgue measure on $\R$. Here  $\mathscr{L}_{\mathrm{loc}}$ denotes the set of all functions $\varphi\in L^1_{\mathrm{loc}}\!\big([0,\infty);[0,\infty)\big)$  such that $\lim_{\sigma\to 0}\theta_\varphi(\sigma)=0$, where
$$\theta_\varphi(\sigma)=\sup\big\{{\textstyle\int_{\textit{I\!I}}}\;\varphi(\tau)\;d\tau\,\adm{1.95ex}{-0.83ex}\,  \textit{I\!I}\;\tn{is a compact sub-interval of}\; [0,\infty)\;\tn{with}\;\mu(\textit{I\!I})\leq\sigma\big\}.$$
Notice that $L^\infty([0,\infty);[0,\infty))\subset\mathscr{L}_{\mathrm{loc}}$.  We denote the norm in $L^\infty\big([0,\infty);\R^{\scriptscriptstyle N}\big)$ by $\|\cdot\|_{\infty}$.

The set $\G S= \{\,(x,y)\mid y\in S(x)\,\}$ is called a \it{graph} of the set-valued map $S:\R^{\scriptscriptstyle M}\rightsquigarrow\R^{\scriptscriptstyle N}$. A set-valued map $S:I\rightsquigarrow\R^{\scriptscriptstyle N}$ is \it{measurable} if for each open set $O\subset\R^{\scriptscriptstyle N}$ the inverse image  $S^{-1}(O)= \{\,t\in I\mid S(t)\cap O\not=\emptyset\,\}$  is Lebesgue measurable set.  A set-valued map $S:\R^{\scriptscriptstyle M}\rightsquigarrow\R^{\scriptscriptstyle N}$ is \it{lower semicontinuous} in  Kuratowski's sense if for each open set $O\subset\R^{\scriptscriptstyle N}$ the set $S^{-1}(O)$ is open. It is equivalent to  $\forall\,(x,y)\in\G S\;\forall\,x_n\to x\;\exists\,y_n\to y\,:\,y_n\in S\!(x_n)$ for large $n\in\N$.\linebreak  A set-valued map $S:\R^{\scriptscriptstyle M}\rightsquigarrow\R^{\scriptscriptstyle N}$ taking nonempty and compact values, is considered\linebreak continuous or Lipschitz continuous if it maintains these properties under the Hausdorff metric evaluation.

 Let $\overline{\R}=\R\cup\{\pm\infty\}$ and $\varphi:\R^{\scriptscriptstyle N}\to\overline{\R}$. The sets $\D\varphi=\{\,x\in\R^{\scriptscriptstyle N}\mid\varphi(x)\not=\pm\infty\,\}$, $\G\varphi=\{\,(x,r)\in\R^{\scriptscriptstyle N}\times\R\mid\varphi(x)=r\,\}$, and $\E\varphi=\{\,(x,r)\in\R^{\scriptscriptstyle N}\times\R\mid\varphi(x)\leq r\,\}$ are called the \emph{effective domain}, the \it{graph} and the \it{epigraph} of $\varphi$, respectively. We say that $\varphi$ is \it{proper} if it never takes  the value $-\infty$ and it is not identically equal to $+\infty$.
 The \it{subdifferential} of the function $\varphi:\R^{\scriptscriptstyle N}\to\overline{\R}$ at the point $x\in\D\varphi$ is the possibly empty set defined by
\begin{equation*}
\partial\varphi(x)=\Big\{p\in\R^{\scriptscriptstyle N}\,\adm{2.56ex}{-1.47ex}\,\liminf_{y\to x}\frac{\varphi(y)-\varphi(x)-\langle p,y-x\rangle}{|y-x|}\geq 0\Big\}.
\end{equation*}
The \it{tangent cone} to the subset $C$ of $\R^{\scriptscriptstyle N}$ at the point $x\in C$ is defined by
\begin{equation*}
T_{C}(x)=\Big\{\,\zeta\in\R^{\scriptscriptstyle N}\,\adm{2.56ex}{-1.47ex}\,\liminf_{\tau\to 0+}\,\frac{\dist(x+\tau \zeta,C)}{\tau}=0 \,\Big\}.
\end{equation*}
The \it{regular normal cone} to the subset $C$ of $\R^{\scriptscriptstyle N}$ at the point $x\in C$ can be defined as
\begin{equation*}
N_C(x)=\big\{\,\xi\in\R^{\scriptscriptstyle N}\,\adm{1.95ex}{-0.83ex}\,\langle \zeta,\xi\rangle\leq 0\;\tn{for all}\;\zeta\in T_{C}(x)\,\big\}.
\end{equation*}
It follows from \cite[Thm. 8.9]{R-W} that $p\in\partial\varphi(x)$ if and only if $(p,-1)\in N_{\E\varphi}(x,\varphi(x))$.\\ By the definition of the regular normal cone for all $(p,q)\in N_{\E\varphi}(x,\varphi(x))$ one has $q\leq 0$.

\noindent The \it{limiting normal cone} to the subset $C$ of $\R^{\scriptscriptstyle N}$ at the point $x\in C$ is defined by
\begin{equation*}
\textit{I\!N}_C(x)=\left\{\,\xi\in\R^{\scriptscriptstyle N}\,\adm{2.5ex}{-1.41ex} \,\exists\;x_n\xrightarrow{\raisebox{-2.2ex}[0pt][0pt]{$\scriptscriptstyle\;C\;$}}x,\;\xi_n\to\xi\;\;\tn{with}\;\;\xi_n\in N_C(x_n)\,\right\}\!,
\end{equation*}
where $x_n\xrightarrow{\raisebox{-2.2ex}[0pt][0pt]{$\scriptscriptstyle\;C\;$}}x$ denotes the convergence in $C$.

A set-valued map $P:I\rightsquigarrow\R^{\scriptscriptstyle M}$ is \it{locally absolutely continuous} if it takes nonempty closed images and for any $[S,T]\subset[0,\infty)$, every $\varepsilon>0$, and any
compact $K\subset\R^{\scriptscriptstyle M}$ there exists $\delta>0$ such that for any finite partition $S\leq s_1<t_1\leq s_2<t_2\leq\cdots\leq s_m<t_m\leq T$
$$\sum\nolimits_{i=1}^{m}(t_i-s_i)<\delta \;\Longrightarrow\; \sum\nolimits_{i=1}^md_K(P(t_i),P(s_i))<\varepsilon,$$ where $d_K(U,D)=\inf\{\,\varepsilon\geq 0\,\mid\,U\cap K\subset D+\varepsilon\B,\, D\cap K\subset U+\varepsilon\B\,\}$ for nonempty $U,D\subset\R^{\scriptscriptstyle M}$.

A sequence of functions $\varphi_n:\R^{\scriptscriptstyle N}\to\overline{\R}$, is said to \it{lower epi-converge} to  function\linebreak $\varphi:\R^{\scriptscriptstyle N}\to\overline{\R}$ (e-$\liminf_{n\to\infty}\varphi_n=\varphi$ for short) if, for every point $x\in\R^n$,
\begin{enumerate}[leftmargin=7mm]
\item[(i)] $\liminf_{n\to\infty}\varphi_n(x_n)\geq\varphi(x)$ for every sequence $x_n\to x$,
\item[(ii)] $\limsup_{i\to\infty}\varphi_{n_i}(x_i)\leq\varphi(x)$ for some sequence $x_i\!\to\! x$ and some increasing sequence~$\{n_i\}$.
\end{enumerate}

\vspace{-4mm}

\section{Representation of Convex Hamiltonian}\label{sec-rch}

\noindent Let $H(t,x,\cdot)$ be a real-valued and convex function defined on $\R^{\scriptscriptstyle N}\!\!\,$  for $t\in[0,\infty)$ and $x\in \R^{\scriptscriptstyle N}\!\!$. By $H^{\ast}(t,x,\cdot)$ we denote the Legendre-Fenchel conjugate of $H(t,x,\cdot)$  with respect to the last variable (in our case $H^{\ast}(t,x,\cdot)$ is an extended real-valued function):
\begin{equation}\label{tlf-1}
H^{\ast}(t,x,v):= \sup\nolimits_{p\,\in\,\R^{\scriptscriptstyle N}}\,\{\,\langle v,p\rangle-H(t,x,p)\,\}.
\end{equation}
 By  a standard property of the Legendre-Fenchel tranform (see \!\cite[Thm.\! 11.1]{R-W}) we obtain
\begin{equation}\label{tlf-2}
H(t,x,p)= \sup\nolimits_{p\,\in\,\R^{\scriptscriptstyle N}}\,\{\,\langle p,v\rangle-H^{\ast}(t,x,v)\,\}.
\end{equation}
A triple $(U,f,l)$ is a \it{representation} of  $H$ if \eqref{def-rh0} holds true for $p\in \R^{\scriptscriptstyle N}$ , where $U=U(t)$ is a nonempty  subset of $\R^{\scriptscriptstyle M}$, $f=f(t,x,\cdot)$ is a function defined on $\R^{\scriptscriptstyle M}$ with values in $\R^{\scriptscriptstyle N}$ and $l=l(t,x,\cdot)$ is a real-valued function  defined on $\R^{\scriptscriptstyle M}$. \\
If $(U,f,l)$ is a representation of  $H$ then by \cite[Prop. 4.1]{AM} we obtain
\begin{equation}\label{rep-inc}
(f,l)(t,x,U(t))\subset\E H^{\ast}(t,x,\cdot),
\end{equation}
where $(f,l)(t,x,u):=(f(t,x,u),l(t,x,u))$. In general, the reverse  implication does not hold. However, if we have
\begin{equation}\label{def-grh}
\G H^{\ast}(t,x,\cdot)\subset(f,l)(t,x,U(t))\subset\E H^{\ast}(t,x,\cdot),
\end{equation}
then, in view of \cite[Prop. 5.7]{AM}, the triple $(U,f,l)$ is a representation of  $H$. The triple $(U,f,l)$ satisfying \eqref{def-grh} is called an \it{epigraphical representation} of  $H$.  The triple $(U,f,l)$ is called a \it{graphical representation of} $H$ if $\G H^{\ast}(t,x,\cdot)=(f,l)(t,x,U(t))$.

The Hamilton-Jacobi equations considered in \cite{B-F-2019} was related to the control problem given by the triple $(U,f,l)$ and a constraints set $A$, where the set $A\subset\R^{\scriptscriptstyle N}$ is nonempty and closed.
 In  \cite{B-F-2019} the assumptions on $(U,f,l)$ and $A$ was denoted by  $\tn{(h)}''$ and  $\tn{(OPC)}$, where
\begin{enumerate}[leftmargin=15mm]
\item[$\pmb{\tn{(h)}'':}$]
A set-valued map $U:[0,\infty)\rightsquigarrow\R^{\scriptscriptstyle M}$  is  measurable  with nonempty closed  images. Let $f:[0,\infty)\times\R^{\scriptscriptstyle N}\times\R^{\scriptscriptstyle M}\to\R^{\scriptscriptstyle N}$ and $l:[0,\infty)\times\R^{\scriptscriptstyle N}\times\R^{\scriptscriptstyle M}\to\R$ be such that
\begin{enumerate}[leftmargin=9mm]
\item[\tn{\bf{(h1)}}]  $f(t,x,u)$ and $l(t,x,u)$ are Lebesgue measurable in $\,t\,$ for all $(x,u)\in\R^{\scriptscriptstyle N}\times\R^{\scriptscriptstyle M}$ and continuous in $(x,u)$ for all $t\in[0,\infty)$ and there exists $\phi(\cdot)\in L^1([0,\infty);\R)$ such that $l(t,x,u)\geq \phi(t)\,$ for  all $\,t\in[0,\infty)$, $x\in\R^{\scriptscriptstyle N}$, $u\in\R^{\scriptscriptstyle M}$\tn{;}
\item[\tn{\bf{(h2)}}] $|f(t,x,u)|+|l(t,x,u)|\,\leq\, c(t)(1+|x|)\;$ for  all $\;t\in[0,\infty)$, $x\in\R^{\scriptscriptstyle N}$, $u\in U(t)$, and\linebreak some  function $c(\cdot)\in L^1_{\mathrm{loc}}\!\big([0,\infty);[0,\infty)\big)$\tn{;}
\item[\tn{\bf{(h3)}}] The set-valued map $x\rightsquigarrow (f,l)(t,x,U(t))\,$ is continuous with closed images for all $t\in[0,\infty)$\tn{;}
\item[\tn{\bf{(h4)}}]  $\forall\,t\in[0,\infty),\,x\in\R^{\scriptscriptstyle N}$ the set  $\{(f(t,x,u),l(t,x,u)+r)\mid u\in U(t),\,r\geq 0\}$ is convex\tn{;}
\item[\tn{\bf{(h5)}}] $|f(t,x,u)|+|l(t,x,u)|\leq q(t)\,$ for  all $\,t\in[0,\infty)$, $x\in\mathrm{bd}\,A$, $u\in U(t)$, and some\\  function $q(\cdot)\in \mathscr{L}_{\tn{loc}}$, where $A$ is a given nonempty, closed subset of $\R^{\scriptscriptstyle N}$\tn{;}
\item[\tn{\bf{(h6)}}] $|f(t,x,u)-f(t,y,u)|+|l(t,x,u)-l(t,y,u)|\leq k(t)\,|x-y|\,$ for all\\ $t\in[0,\infty)$, $x,y\in\R^{\scriptscriptstyle N}$, $u\in U(t)$, and some function $k(\cdot)\in \mathscr{L}_{\tn{loc}}$.
\end{enumerate}
\item[\tn{\bf{(OPC):}}]There exist $\eta>0$, $r>0$, $M\geq 0$ such that for almost all $t\in[0,\infty)$ and any\linebreak $y\in\mathrm{bd}\,A+\eta\B$, and any $v\in f(t,y,U(t))$ with $\inf_{n\in \textit{I\!N}^1_{y,\eta}}\langle n,v\rangle\leq 0$, we can find\linebreak $w\in f(t,y,U(t))\cap\B(v,M)$ satisfying
$$\inf\nolimits_{\,n\in\textit{I\!N}^1_{y,\eta}}\min\big\{\,\langle n,w\rangle,\;\langle n,w-v\rangle\,\big\}\geq r,$$
where $\textit{I\!N}^1_{y,\eta}:=\big\{n\in\mathds{S}\,\adm{1.95ex}{-0.83ex}\,n\in\mathrm{cl}\,\mathrm{con}\,\textit{I\!N}_A(x),\,x\in\mathrm{bd}\,A\cap\B(y,\eta)\big\}$.
\end{enumerate}
We provide properties of the Hamiltonian corresponding to the triple $(U,f,l)$ satisfying
$\tn{(h)}''$ and  $\tn{(OPC)}$. The properties of the corresponding Hamiltonian are denoted by
\begin{enumerate}[leftmargin=17mm]
\item[$\pmb{\tn{(h)}_H'':}$] Let the Hamiltonian $H:[0,\infty)\times\R^{\scriptscriptstyle N}\times\R^{\scriptscriptstyle N}\to\R$ be such that
\begin{enumerate}[leftmargin=9.4mm]
\item[\bf{(H1)}] $H(t,x,p)$ is  measurable with respect to the first variable, continuous with respect to the second variable and convex with respect to the third variable.
\item[\bf{(H2)}] $H(t,x,0)\!\leq\!-\phi(t)$ for all $t\!\in\![0,\!\infty)$, $\!x\!\in\!\R^{\scriptscriptstyle N}$ and some function $\phi\!\in\! L^1\!([0,\!\infty);\R)$.
\item[]\hspace{-1cm} There exists a function $\lambda:[0,\infty)\times\R^{\scriptscriptstyle N}\to[0,\infty)$ that is measurable with respect \item[]\hspace{-1cm} to the first variable and continuous with respect to the second variable and there
\item[]\hspace{-1cm}  exist $c(\cdot)\in L^1_{\mathrm{loc}}\!\big([0,\infty);[0,\infty)\big)$  and $k(\cdot),q(\cdot)\in \mathscr{L}_{\tn{loc}}$ such that
\item[\bf{(H3)}] $\forall\,x\in\R^{\scriptscriptstyle N}\;\lambda(t,x)\leq c(t)(1+|x|)$ and $\forall\,x\in\mathrm{bd}\,A\;\lambda(t,x)\leq q(t)$ for all $t\in[0,\infty)$;
\item[\bf{(H4)}] $|\lambda(t,x)-\lambda(t,y)|\leq k(t)|x-y|$ for all $t\in[0,\infty)$,  $x,y\in\R^{\scriptscriptstyle N}$;
\item[\bf{(H5)}] $|H(t,x,p)-H(t,y,p)|\leq k(t)(1+|p|)|x-y|$ for all $t\in[0,\infty)$,  $x,y,p\in\R^{\scriptscriptstyle N}$;
\item[\bf{(H6)}] $|H(t,x,p)-H(t,x,q)|\leq\lambda(t,x)|p-q|$ for all $t\in[0,\infty)$, $x,p,q\in\R^{\scriptscriptstyle N}$;
\item[\bf{(H7)}] $|H^{\ast}(t,x,v)|\leq \lambda(t,x)$ for all  $v\in\D H^{\ast}(t,x,\cdot)$, $x\in\R^{\scriptscriptstyle N}$, $t\in[0,\infty)$;\vspace{1mm}
\end{enumerate}
\item[$\pmb{\tn{(OPC)}_H\!:}$] There exist $\eta>0$, $r>0$, $M\geq 0$ such that for almost all $t\in[0,\infty)$ and any $y\in\mathrm{bd}\,A+\eta\B$, and any $v\in\D H^{\ast}(t,y,\cdot)$ with $\inf_{n\in \textit{I\!N}^1_{y,\eta}}\langle n,v\rangle\leq 0$, we can find $w\in\D H^{\ast}(t,y,\cdot)\cap\B(v,M)$ satisfying
$$\inf\nolimits_{\,n\in\textit{I\!N}^1_{y,\eta}}\min\big\{\,\langle n,w\rangle,\;\langle n,w-v\rangle\,\big\}\geq r,$$
where $\textit{I\!N}^1_{y,\eta}:=\big\{n\in\mathds{S}\,\adm{1.95ex}{-0.83ex}\,n\in\mathrm{cl}\,\mathrm{con}\,\textit{I\!N}_A(x),\,x\in\mathrm{bd}\,A\cap\B(y,\eta)\big\}$.
\end{enumerate}

\begin{Th}\label{rwwar-i}
If  the control system given by a triple $(U,f,l)$ satisfies  $\tn{(h)}''$, then the\linebreak corresponding Hamiltonian $H(t,x,p)$ given by \eqref{def-rh0} meets the criteria  $\tn{(h)}_H''$. Moreover, if $\tn{(OPC)}$ is satisfied, then $\tn{(OPC)}_H$ also holds. Additionally, $(U,f,l)$ is an epigraphical representation of $H$, i.e \eqref{def-grh} holds for all $t\in[0,\infty)$, $x\in\R^{\scriptscriptstyle N}$.
\end{Th}
In the proof of Theorem \ref{rwwar-i} we use the following Lemmas.
\begin{Lem}\label{rep-rde}
Let $H(t,x,\cdot)$ be a real-valued and convex function defined on $\R^{\scriptscriptstyle N}$. Assume that $U(t)$ is a nonempty  subset of $\R^{\scriptscriptstyle M}$. Furthermore, let $f(t,x,\cdot)$ be a function defined on $\R^{\scriptscriptstyle M}$ with values in $\R^{\scriptscriptstyle N}$, and the set $f(t,x,U(t))$ be closed and convex. Moreover, let $l(t,x,\cdot)$ be a  real-valued function defined on $\R^{\scriptscriptstyle M}$.  If  $(U,f,l)$ is a representation of $H$, then
\begin{equation}\label{rep-rde-neq}
\D H^{\ast}(t,x,\cdot)\,=\,f(t,x,U(t)).
\end{equation}
\end{Lem}

\begin{proof}
Given that $H(t,x,\cdot)$ is a real-valued and convex function, it follows that $H^{\ast}(t,x,\cdot)$ is a proper function. Moreover, by \eqref{rep-inc} we have
\begin{equation}\label{rep-rde-d1}
H^{\ast}(t,x,f(t,x,u))\leq l(t,x,u),\;\;\;\forall\;u\in U(t).
\end{equation}
 Hence  $f(t,x,U(t))\subset \D H^{\ast}(t,x,\cdot)$. Now we show that $\D H^{\ast}(t,x,\cdot)\subset f(t,x,U(t))$.\linebreak We suppose that this inclusion is false.
Then there exists an element $v\in \D H^{\ast}(t,x,\cdot)$\linebreak and $v\not\in f(t,x,U(t))$. The set $f(t,x,U(t))$ is nonempty, closed and convex, so by the\linebreak Separation Theorem, there exist  $q\in\R^{\scriptscriptstyle N}$ and $\alpha,\beta\in\R$ such that
\begin{equation*}
 \langle\,v,q \,\rangle\leq \alpha<\beta\leq \langle\, f(t,x,u),q\,\rangle,\;\;\;\forall\;u\in U(t).
\end{equation*}
We notice that by the above inequality we obtain
\begin{equation}\label{stod-dlae0101}
 0\,<\,\beta-\alpha\,\leq\, \langle\, f(t,x,u)-v,q\,\rangle,  \;\;\;\forall\;u\in U(t).
\end{equation}
By \eqref{tlf-1} we have $H^{\ast}(t,x,v)\geq -H(t,x,0)$ for all $v\in\R^{\scriptscriptstyle N}$. The latter, together with \eqref{rep-rde-d1}, implies that $l(t,x,u)\geq-H(t,x,0)$ for all $u\in U(t)$. This means that the function $l(t,x,\cdot)$ is bounded from below on the set $U(t)$.
Let us define  $\xi(t,x):=\inf_{u\in U(t)}l(t,x,u)$ and $\varepsilon:=\beta-\alpha$.  Let $n\in\N$ be large enough so that
\begin{equation}\label{stod-dlae01011}
H^{\ast}(t,x,v)-\xi(t,x)\;<\;n(\beta-\alpha).
\end{equation}
Since $H(t,x,\cdot)$ is real-valued, for  $\,p:=-(n+1)q\,$  there exists  $\,u_{\varepsilon}\in U(t)\,$  such that
\begin{equation}\label{stod-dlae01012}
H(t,x,p)-\varepsilon\leq \langle\, p,f(t,x,u_{\varepsilon})\,\rangle-l(t,x,u_{\varepsilon}).
\end{equation}
From  (\ref{stod-dlae01011}), (\ref{stod-dlae01012}) and (\ref{stod-dlae0101}), it follows that
\begin{eqnarray*}
n(\beta-\alpha)
&>& \langle\,v,p\,\rangle - H(t,x,p)-\xi(t,x)\\
&\geq & \langle\, v-f(t,x,u_{\varepsilon}),p\,\rangle-\varepsilon\\
&\geq & (n+1)(\beta-\alpha)-\varepsilon\;=\;n(\beta-\alpha).
\end{eqnarray*}
Thus, we obtain a contradiction, that completes the proof.
\end{proof}

\begin{Lem}\label{rep-epi}
Under the assumptions of Lemma \ref{rep-rde}, if  $(U,f,l)$ is a representation of $H$ with the closed, convex set
$\{(f(t,x,u),l(t,x,u)+r)\mid u\in U(t),\,r\geq 0\}$, then
\begin{eqnarray}\label{rep-epi-neq}
\E H^{\ast}(t,x,\cdot)\,=\,\{(f(t,x,u),l(t,x,u)+r)\mid u\in U(t),\,r\geq 0\}.
\end{eqnarray}
\end{Lem}
\begin{proof}
Let us define $\Gamma(t,x):=\{(f(t,x,u),l(t,x,u)+r)\mid u\in U(t),\,r\geq 0\}$.
Given that $H(t,x,\cdot)$ is a real-valued and convex function, it follows that $H^{\ast}(t,x,\cdot)$ is a proper function.\linebreak Moreover, by \ref{rep-inc} we have $(f,l)(t,x,U(t))\subset\E H^{\ast}(t,x,\cdot)$. Hence
$$(f,l)(t,x,U(t))+\{0\}\times[0,\infty)\subset\E H^{\ast}(t,x,\cdot)+\{0\}\times[0,\infty).$$
Therefore, $\Gamma(t,x)\subset\E H^{\ast}(t,x,\cdot)$. Now we show that $\E H^{\ast}(t,x,\cdot)\subset\Gamma(t,x)$. We suppose\linebreak that this inclusion is false.
Then there exists  $(v,\eta)\in \E H^{\ast}(t,x,\cdot)$ and $(v,\eta)\not\in \Gamma(t,x)$.\linebreak The set $\Gamma(t,x)$ is nonempty, closed and convex, so by the Separation Theorem, there exist  $(q,\tau)\in\R^{\scriptscriptstyle N}\times\R$ and $\alpha\in\R$ such that
\begin{equation}\label{rep-epi-1}
 \langle\,(f(t,x,u),l(t,x,u)+r),(q,\tau) \,\rangle\leq \alpha< \langle\, (v,\eta),(q,\tau)\,\rangle,\;\;\;\forall\;u\in U(t),\,r\geq 0.
\end{equation}
By the inequalities \eqref{rep-epi-1}, it follows that
\begin{equation}\label{rep-epi-2}
 \langle\,f(t,x,u),q\,\rangle+l(t,x,u)\,\tau+r\,\tau\leq \alpha,\;\;\;\forall\;u\in U(t),\,r\geq 0.
\end{equation}
We note that if  $\tau>0$, then  passing to the  limit in \eqref{rep-epi-2} as $r\to\infty$, we obtain that $\infty\leq\alpha$, in contradiction to the fact that $\alpha$ belongs to $\R$. Therefore, $\tau\leq 0$. Now, we show that $\tau$ cannot be equal to zero. Suppose that $\tau=0$. Then, by \eqref{rep-epi-1}, we get
\begin{equation}\label{rep-epi-3}
\langle\,f(t,x,u),q\,\rangle\leq \alpha< \langle\, v,q\,\rangle,\;\;\;\forall\;u\in U(t).
\end{equation}
Since $(v,\eta)\in \E H^{\ast}(t,x,\cdot)$, in particular $v\in\D H^{\ast}(t,x,\cdot)$. In view of Lemma \ref{rep-rde}, we have $f(t,x,U(t))=\D H^{\ast}(t,x,\cdot)$. Therefore, there exists $\tilde{u}\in U(t)$ such that $f(t,x,\tilde{u})=v$. By \eqref{rep-epi-3} we obtain $\langle\, v,q\,\rangle=\langle\,f(t,x,\tilde{u}),q\,\rangle\leq \alpha< \langle\,v,q\,\rangle$, which leads to a contradiction. Consequently, we get $\tau<0$. Substituting $r:=0$ into \eqref{rep-epi-1} and dividing by $|\tau|$, we obtain
\begin{equation}\label{rep-epi-4}
 \langle\,f(t,x,u),\hat{q} \,\rangle-l(t,x,u)\leq \alpha< \langle\,v,\hat{q}\,\rangle-\eta,\;\;\;\forall\;u\in U(t),
\end{equation}
where $\hat{q}:=q/|\tau|$. Since $(U(t),f,l)$ is a representation of $H$, by \eqref{rep-epi-4} we have $H(t,x,\hat{q})\leq\alpha$.
Since $(v,\eta)\in \E H^{\ast}(t,x,\cdot)$, in particular $H^{\ast}(t,x,v)\leq\eta$.  The latter, together with \eqref{tlf-2}  and \eqref{rep-epi-4}, implies that $\alpha<H(t,x,\hat{q})$.  Consequently, we get $H(t,x,\hat{q})\leq\alpha<H(t,x,\hat{q})$. Thus, we obtain a contradiction, that completes the proof.
\end{proof}

\begin{proof}[Proof of Theorem \ref{rwwar-i}.]
 We define $\lambda(\cdot,\cdot)$ on the set $[0,\infty)\times\R^{\scriptscriptstyle N}$  by the formula
\begin{equation}\label{deflambda}
 \lambda(t,x):=\sup\nolimits_{u\in U(t)}\{\,|f(t,x,u)|+|l(t,x,u)|\,\}.
\end{equation}
By (h2) it follows that $\lambda(\cdot,\cdot)$ is a non-negative real-valued map on the set  $[0,\infty)\times\R^{\scriptscriptstyle N}$.
Since $(t,u)\to (f,l)(t,x,u)$ is a Carathéodory function for all $x\in\R^{\scriptscriptstyle N}$, by \cite[Thm. 8.2.11]{A-F} the map $t\to\lambda(t,x)$ is measurable for all $x\in\R^{\scriptscriptstyle N}$. We observe that (H3) follows directly from (h2) and (h5). It is not difficult to show that (H4) follows from (h6), in particular the map $x\to\lambda(t,x)$ is continuous for all $t\in[0,\infty)$.

Let $H$ be given by \eqref{def-rh0}. By (h2) it follows that $p\to H(t,x,p)$ is a convex real-valued map on   $\R^{\scriptscriptstyle N}$ for all $t\in[0,\infty)$ and $x\in\R^{\scriptscriptstyle N}$.
Since $(t,u)\to \langle\, p,f(t,x,u)\,\rangle-l(t,x,u)$ is a\linebreak Carathéodory map for all $x,p\in\R^{\scriptscriptstyle N}$, in view of \cite[Thm. 8.2.11]{A-F}, the map $t\to H(t,x,p)$ is measurable for all $x,p\in\R^{\scriptscriptstyle N}$. It is not difficult to show that (H5) follows from (h6) and (H6) follows from \eqref{deflambda}. We observe that (H5) and (H6) imply the map $(x,p)\to H(t,x,p)$ is locally Lipschitz continuous for all  $t\in[0,\infty)$. In particular, this map is continuous for all  $t\in[0,\infty)$. Thus, the condition (H1) is satisfied.

By (h2) and (h3) we get that the set $(f,l)(t,x,U(t))$ is compact for all $t\in[0,\infty)$, $x\in\R^{\scriptscriptstyle N}$. The latter, together with (h4), implies that the set $\{(f(t,x,u),l(t,x,u)+r)\mid u\in U(t),\,r\geq 0\}$\linebreak is  closed and convex for all $t\in[0,\infty)$, $x\in\R^{\scriptscriptstyle N}$. Additionally, the set $f(t, x, U(t))$ is\linebreak compact and convex, as it is the projection of the above two sets with these properties.\linebreak Consequently, we obtain that the assumptions of Lemma \ref{rep-rde} are satisfied. Therefore, by Lemmas \ref{rep-rde} and \ref{rep-epi}, for all $t\in[0,\infty)$, $x\in\R^{\scriptscriptstyle N}$, we get that  \eqref{rep-rde-neq} and \eqref{rep-epi-neq} hold.

We show that (H7) holds. Let us fix arbitrarily $t\in[0,\infty)$, $x\in\R^{\scriptscriptstyle N}$. Let $v\in\D H^{\ast}(t,x,\cdot)$. By  \eqref{rep-rde-neq} there exists $u\in U(t)$ such that $v=f(t,x,u)$. The latter, together with \eqref{rep-inc}, implies that  $H^{\ast}(t,x,v)=H^{\ast}(t,x,f(t,x,u))\leq l(t,x,u)$. Therefore,\vspace{-1mm}
\begin{equation}\label{rwwar-i-1}
H^{\ast}(t,x,v)\,\leq\,\lambda(t,x),\;\;\forall\,v\in\D H^{\ast}(t,x,\cdot).\vspace{-1mm}
\end{equation}
In view of \eqref{tlf-1} we have $H^{\ast}(t,x,v)\geq -H(t,x,0)$ for all $v\in\R^{\scriptscriptstyle N}$. Moreover, in view of \eqref{def-rh0} we obtain $-H(t,x,0)=\inf_{u\in U(t)}l(t,x,u)$. Therefore, $H^{\ast}(t,x,v)\,\geq\,\inf\nolimits_{u\in U(t)}l(t,x,u)$ for all $v\in\R^{\scriptscriptstyle N}$.
This inequality, along with $l(t,x,u)\geq -\lambda(t,x)$ for all $u\in U(t)$, implies that\vspace{-1mm}
\begin{equation}\label{rwwar-i-3}
H^{\ast}(t,x,v)\,\geq\,-\lambda(t,x),\;\;\forall\,v\in\R^{\scriptscriptstyle N}.\vspace{-1mm}
\end{equation}
Combining inequalities \eqref{rwwar-i-1} and \eqref{rwwar-i-3} we obtain (H7). By (h1), we have $l(t,x,u)\geq \phi(t)\,$ for  all  $u\in\R^{\scriptscriptstyle M}$. The latter, together with $H(t,x,0)=\sup_{u\in U(t)}\{-l(t,x,u)\}$, implies (H2).

Since $f(t,x,U(t))=\D H^{\ast}(t,x,\cdot)$ for all $t\in[0,\infty)$ and $x\in\R^{\scriptscriptstyle N}$, the condition $\tn{(OPC)}_H$ follows directly from the condition (OPC).

It remains to be proven \eqref{def-grh}. Let us fix arbitrarily $t\in[0,\infty)$, $x\in\R^{\scriptscriptstyle N}$. Since $(U,f,l)$ is a representation of $H$, by \eqref{rep-inc} we have $(f,l)(t,x,U(t))\subset\E H^{\ast}(t,x,\cdot)$. Now we show that
$\G H^{\ast}(t,x,\cdot)\subset(f,l)(t,x,U(t))$. Let $(v,\eta)\in\G H^{\ast}(t,x,\cdot)$. Then, from the definition of the graph, we have $\eta=H^{\ast}(t,x,v)$ and $(v,\eta)\in\E H^{\ast}(t,x,\cdot)$. The latter, together with  \eqref{rep-epi-neq}, implies that there exist $u\in U(t)$ and $r\geq 0$ such that $v=f(t,x,u)$ and $\eta=l(t,x,u)+r$. Since  $(f,l)(t,x,U(t))\subset\E H^{\ast}(t,x,\cdot)$, it follows that $H^{\ast}(t,x,f(t,x,u))\leq l(t,x,u)$. Therefore,\vspace{-1mm}
$$r+l(t,x,u)=\eta=H^{\ast}(t,x,v)=H^{\ast}(t,x,f(t,x,u))\leq l(t,x,u).\vspace{-1.5mm}$$
From the above inequality, we obtain that $r=0$. Thus, $(v,\eta)=(f,l)(t,x,u)\in (f,l)(t,x,U(t))$. Consequently, the property \eqref{def-grh} holds true.
\end{proof}

Our aim is to construct a representation $(U,f,l)$ of the Hamiltonian $H$ such that $(U,f,l)$ satisfies $\tn{(h)}''$ and $\tn{(OPC)}$, provided that $H$ satisfies $\tn{(h)}''_H$ and $\tn{(OPC)}_H$. A preliminary result is therefore required. We define  $F:[0,\infty)\times\R^{\scriptscriptstyle N}\rightsquigarrow\R^{\scriptscriptstyle N}$ and $E:[0,\infty)\times\R^{\scriptscriptstyle N}\rightsquigarrow\R^{\scriptscriptstyle N+1}$ by
\begin{equation}\label{def-FE}
F(t,x)=\D H^{\ast}(t,x,\cdot)\qquad\tn{and}\qquad E(t,x)=\E H^{\ast}(t,x,\cdot).
\end{equation}
Using Propositions 2.5 and 2.6 in \cite{AM}, and Chapters 5 and 14 in \cite{R-W}, we obtain:

\begin{Cor}\label{wrow-wm} Assume that $H:[0,\infty)\times\R^{\scriptscriptstyle N}\times\R^{\scriptscriptstyle N}\to\R$ satisfies \tn{(H1)}. Then
\begin{enumerate}
\item[\tn{\bf{(M1)}}] $F(t,x)$ is a nonempty, convex subset of $\;\R^{\scriptscriptstyle N}$ for all $t\in[0,\infty)$, $x\in\R^{\scriptscriptstyle N}$\tn{;}
\item[\tn{\bf{(M2)}}] $E(t,x)$ is a nonempty, closed, convex subset of $\;\R^{\scriptscriptstyle N+1}$ for all $t\in[0,\infty)$, $x\in\R^{\scriptscriptstyle N}$\tn{;}
\item[\tn{\bf{(M3)}}] $x\to F(t,x)$ is lower semicontinuous for all  $t\in[0,\infty)$\tn{;}
\item[\tn{\bf{(M4)}}] $x\to E(t,x)$ has a closed graph  and is lower semicontinuous for all $t\in[0,\infty)$\tn{;}
\item[\tn{\bf{(M5)}}] $t\to F(t,x)$ and $t\to E(t,x)$ are measurable for all $x\in\R^{\scriptscriptstyle N}$\tn{.}\vspace{1mm}
\item[]\hspace{-1.3cm}Additionally, if $H$ satisfies \tn{(H5)} with $k:[0,\infty)\to[0,\infty)$, then\vspace{1mm}
\item[\tn{\bf{(M6)}}] $\textit{d\!l}_\mathcal{H}(F(t,x),F(t,y))\leq k(t)\,|x-y|\,$ for all $t\in[0,\infty)$, $x,y\in\R^{\scriptscriptstyle N}$\tn{;}
\item[\tn{\bf{(M7)}}] $\textit{d\!l}_\mathcal{H}(E(t,x),E(t,y))\leq 2\,k(t)\,|x-y|\,$ for all $t\in[0,\infty)$, $x,y\in\R^{\scriptscriptstyle N}$\tn{.}\vspace{1mm}
\item[]\hspace{-1.3cm}Additionally, if $H$ satisfies \tn{(H6)} and \tn{(H7)} with $\lambda:[0,\infty)\times\R^{\scriptscriptstyle N}\to[0,\infty)$ being a continuous
\item[]\hspace{-1.3cm}function with respect to the second variable, then\vspace{1mm}
\item[\tn{\bf{(M8)}}] $F(t,x)$ is a nonempty, compact, convex subset of $\;\R^{\scriptscriptstyle N}$ for all $t\in[0,\infty)$, $x\in\R^{\scriptscriptstyle N}$\tn{;}
\item[\tn{\bf{(M9)}}]  $x\to F(t,x)$  is continuous in the sense of the Hausdorff metric  for all  $t\in[0,\infty)$\tn{;}
\item[\tn{\bf{(M10)}}] $\|F(t,x)\|\leq\lambda(t,x)$ for all $t\in[0,\infty)$, $x\in\R^{\scriptscriptstyle N}$\tn{.}
\end{enumerate}
\end{Cor}

To construct this representation, we adopt the method outlined by Misztela in \cite{AM}. The specific nature of conditions $\tn{(h)}''\!\!$ and $\!\tn{(OPC)}$ necessitates modifications to the \mbox{constructions} used in the proofs of Theorems 5.6 and 5.8 in \cite{AM}. In particular, this adjustment concerns a change to align the definition of the function $\omega(\cdot,\cdot)$ in the proof of Theorem 5.8 with the definition in this paper, as presented by formula \eqref{def-omega}. In Theorem 5.6 in \cite{AM}\linebreak it is assumed that  $\omega(t,x) \geq 1$, but its conclusion remains valid when $\omega(t,x) \geq 0$.\linebreak Indeed, if $\omega(t,x) = 0$, then in the proof of Theorem 5.6, one should take $a:=0$ instead of $a:=z/\omega(t,x)$. The following two theorems are crucial to achieve this representation:

\begin{Th}\label{rep-t}
Assume that $\tn{(h)}''_H$ holds. Then there exist  $f:[0,\infty)\times\R^{\scriptscriptstyle N}\times\R^{\scriptscriptstyle N+1}\to\R^{\scriptscriptstyle N}$ and $l:[0,\infty)\times\R^{\scriptscriptstyle N}\times\R^{\scriptscriptstyle N+1}\to\R$,  measurable in $\,t\,$ for all $(x,u)\in\R^{\scriptscriptstyle N}\times\R^{\scriptscriptstyle N+1}$ and continuous in $(x,u)$ for all $t\in[0,\infty)$, such that for every $t\in[0,\infty)$, $x,p\in\R^{\scriptscriptstyle N}$
\begin{equation*}
 H(t,x,p)=\sup\nolimits_{u\in \B}\,\{\,\langle\, p,f(t,x,u)\,\rangle-l(t,x,u)\,\}
\end{equation*}
and $f(t,x,\B)=\D H^{\ast}(t,x,\cdot)$, where $\B$ is the closed unit ball in $\R^{\scriptscriptstyle N+1}$.
Additionally,
\begin{enumerate}
\item[\tn{\bf{(A1)}}] For all $t\in[0,\infty)$, $x,y\in\R^{\scriptscriptstyle N}$, $u\in\B$
\begin{equation*}
|f(t,x,u)-f(t,y,u)|+|l(t,x,u)-l(t,y,u)|\leq 40\,(N+1)\,k(t)\,|x-y|.
\end{equation*}
\item[\tn{\bf{(A2)}}] $|f(t,x,u)|+|l(t,x,u)|\leq 20\lambda(t,x)\,$ for  all $\,t\in[0,\infty)$, $x\in\R^{\scriptscriptstyle N}$, $u\in\B$.
\item[\tn{\bf{(A3)}}] $\G H^{\ast}(t,x,\cdot)\subset(f,l)(t,x,\B)\subset\E H^{\ast}(t,x,\cdot)\,$ for  all $\,t\in[0,\infty)$, $x\in\R^{\scriptscriptstyle N}$,\\ i.e. $(\B,f,l)$ is an epigraphical representation of $H$.
\item[\tn{\bf{(A4)}}] $l(t,x,u)\geq \phi(t)\,$ for  all $\,t\in[0,\infty)$, $x\in\R^{\scriptscriptstyle N}$, $u\in\R^{\scriptscriptstyle N+1}$.
\item[\tn{\bf{(A5)}}] Furthermore, if $H$ and $\lambda$ are continuous, so are $f$, $l$.
\end{enumerate}
\end{Th}

\begin{proof}
By Corollary \ref{wrow-wm}, the set-valued maps $F$ and $E$ given by \ref{def-FE} satisfy (M1)-(M10).

We define $\mathrm{e}:[0,T]\times\R^{\scriptscriptstyle N}\times\R^{\scriptscriptstyle N+1}\to\R^{\scriptscriptstyle N+1}$ by the formula
\begin{equation}\label{rep-t-01}
\mathrm{e}(t,x,u):=S_{\!\scriptscriptstyle N+1}\big[E(t,x)\cap\B\big(\omega(t,x)\,u,2\,\dist(\omega(t,x)\,u,E(t,x))\big)\big],
\end{equation}
where $S_{\!\scriptscriptstyle N+1}[\,\cdot\,]$ is the Steiner selection defined as in \cite[p. 365]{A-F} and
\begin{equation}\label{def-omega}
\omega(t,x):=2\lambda(t,x).
\end{equation}
Using arguments similar to the ones used in
the proof of  Theorem 5.6 in \cite{AM}, we show that $\mathrm{e}(t,x,u)$ is  measurable in $\,t\,$  and continuous in $(x,u)$. Moreover,
\begin{equation}\label{stw-inql}
[E(t,x)\cap\omega(t,x)\B]\subset\mathrm{e}(t,x,\B)\subset\mathrm{e}(t,x,\R^{\scriptscriptstyle N+1})\subset E(t,x),
\end{equation}
for all $t\in[0,\infty)$, $x\in\R^{\scriptscriptstyle N}$. By (M10) and (H7), for all $t\in[0,\infty)$, $x\in\R^{\scriptscriptstyle N}$, we obtain
\begin{equation}\label{stw-inql-1}
\|\G H^{\ast}(t,x,\cdot)\|\leq\|F(t,x)\|+\sup\nolimits_{v\in F(t,x)}|H^{\ast}(t,x,v)|\leq\omega(t,x).
\end{equation}
Combining \eqref{stw-inql} and \eqref{stw-inql-1}, for all $t\in[0,\infty)$, $x\in\R^{\scriptscriptstyle N}$,  we get
\begin{equation}\label{stw-inql-2}
\G H^{\ast}(t,x,\cdot)\subset\mathrm{e}(t,x,\B)\subset\mathrm{e}(t,x,\R^{\scriptscriptstyle N+1})\subset\E H^{\ast}(t,x,\cdot).
\end{equation}
By \cite[Lem. 5.1]{AM} and (M7), for all $t\in[0,\infty)$, $x,y\in\R^{\scriptscriptstyle N}$, $u\in\B$, we obtain
\begin{eqnarray}\label{stw-inql-3}
|\mathrm{e}(t,x,u)-\mathrm{e}(t,y,u)| &\leq & 5(N+1)[\,\textit{d\!l}_\mathcal{H}(E(t,x),E(t,y))+|\omega(t,x)u-\omega(t,y)u|\,]\nonumber\\
&\leq & 5(N+1)[\,2k(t)|x-y|+2|\lambda(t,x)-\lambda(t,y)|\,]\nonumber\\
&\leq & 20(N+1)k(t)|x-y|.
\end{eqnarray}

In view of \cite[p. 366]{A-F} we have $S_{\!\scriptscriptstyle N+1}[C]\in C$ for any nonempty, convex and compact subset $C$ of $\R^{\scriptscriptstyle N+1}$. Thus, in view of \eqref{rep-t-01}, for all $t\in[0,T]$, $x\in\R^{\scriptscriptstyle N}$, $u\in\R^{\scriptscriptstyle N+1}$, we obtain
\begin{equation}\label{rep-t-02}
\mathrm{e}(t,x,u)\in\B\big(\omega(t,x)\,u,2\,\dist(\omega(t,x)\,u,E(t,x))\big).
\end{equation}
Combining \eqref{rep-t-02} and \eqref{stw-inql-1}, for all $t\in[0,\infty)$, $x\in\R^{\scriptscriptstyle N}$, $u\in\B$, we have
\begin{eqnarray}\label{stw-inql-4}
|\mathrm{e}(t,x,u)| &\leq & \omega(t,x)\,|u| +2\,\dist(\omega(t,x)\,u,E(t,x))\nonumber\\
&\leq & 3\,\omega(t,x)\,|u|+2\,\dist(0,E(t,x))\nonumber\\
&\leq & 3\,\omega(t,x)+2\,\|\G H^{\ast}(t,x,\cdot)\|\nonumber\\
&\leq & 5\,\omega(t,x)\;=\;10\,\lambda(t,x).
\end{eqnarray}

We define the functions $f$ and $l$ as components of the function~$\mathrm{e}$, i.e. $(f,l)=\mathrm{e}$. Then, by \eqref{stw-inql-2} and \cite[Prop. 5.7]{AM}, the triple $(\B,f,l)$ is a representation of $H$ and $f(t,x,\B)=\D H^{\ast}(t,x,\cdot)$ for all $t\in[0,\infty)$, $x\in\R^{\scriptscriptstyle N}$.  By \eqref{stw-inql-3}, for all $t\in[0,\infty)$, $x,y\in\R^{\scriptscriptstyle N}$, $u\in\B$,
\begin{equation}\label{stw-inql-flel}
|f(t,x,u)\!-\!f(t,y,u)|\!+\!|l(t,x,u)\!-\!l(t,y,u)|\leq 2|\mathrm{e}(t,x,u)\!-\!\mathrm{e}(t,y,u)|\leq 40(N\!+\!1)k(t)|x\!-\!y|.
\end{equation}
Therefore, the condition (A1) is satisfied. By \eqref{stw-inql-4}, for all $t\in[0,\infty)$, $x\in\R^{\scriptscriptstyle N}$, $u\in\B$,
\begin{equation}\label{stw-inql-fleg}
|f(t,x,u)|+|l(t,x,u)|\leq 2|\mathrm{e}(t,x,u)|\leq 10\omega(t,x)= 20\lambda(t,x).
\end{equation}
Therefore, the condition (A2) is satisfied.
By \eqref{stw-inql-2} we obtain (A3) and moreover,  we
get $H^{\ast}(t,x,f(t,x,u))\leq l(t,x,u)$ for all $t\in[0,\infty)$, $x\in\R^{\scriptscriptstyle N}$, $u\in\R^{\scriptscriptstyle N+1}$. The latter, together with (H2) and \eqref{tlf-1}, implies  $\phi(t)\leq-H(t,x,0)\leq H^{\ast}(t,x,f(t,x,u))\leq l(t,x,u)$ for all $t\in[0,\infty)$, $x\in\R^{\scriptscriptstyle N}$, $u\in\R^{\scriptscriptstyle N+1}$.
Thus,  (A4) is satisfied. Using the same arguments as in the proof of Theorem 5.6 in \cite{AM} we obtain (A5).
\end{proof}

\begin{Th}\label{rwwar}
Assume that $\tn{(h)}_H''$ and  $\tn{(OPC)}_H$ hold. Then, there exists a representation $(\B,f,l)$ of the Hamiltonian $H$, satisfying $\tn{(h)}''$ and  $\tn{(OPC)}$.
\end{Th}
\begin{proof}
In view of Theorem \ref{rep-t}, there exist functions $f:[0,\infty)\times\R^{\scriptscriptstyle N}\times\R^{\scriptscriptstyle N+1}\to\R^{\scriptscriptstyle N}$ and $l:[0,\infty)\times\R^{\scriptscriptstyle N}\times\R^{\scriptscriptstyle N+1}\to\R$,  measurable in $\,t\,$ for all $(x,u)\in\R^{\scriptscriptstyle N}\times\R^{\scriptscriptstyle N+1}$ and continuous in $(x,u)$ for all $t\in[0,\infty)$, such that the triple $(\B,f,l)$ is a representation of $H$ and $f(t,x,\B)=\D H^{\ast}(t,x,\cdot)$ for all $t\in[0,\infty)$, $x\in\R^{\scriptscriptstyle N}$. Additionally, the conditions  (A1)-(A5) from\linebreak Theorem \ref{rep-t} are satisfied.  Let $U:[0,\infty)\rightsquigarrow\R^{\scriptscriptstyle N+1}$ be defined such that $U(\cdot)\equiv\B$ and $M=N+1$. Then, we show that the triple $(U,f,l)$ satisfies the conditions $\tn{(h)}''$ and  $\tn{(OPC)}$. We observe that (h1) follows directly from (A4) and the aforementioned properties of the functions  $f$ and $l$. Similarly, (h2) and (h5) are a direct consequence of (A2) and (H3). Since  $u\to (f,l)(t,x,u)$ is continuous  for all $t\in[0,\infty)$, $x\in\R^{\scriptscriptstyle N}$, and $\B$ is a compact set,\linebreak the set $(f,l)(t,x,\B)$ is compact for all $t\in[0,\infty)$, $x\in\R^{\scriptscriptstyle N}$. It is not difficult to show that
\begin{equation}\label{rwwar-1}
\textit{d\!l}_\mathcal{H}((f,l)(t,x,\B),(f,l)(t,y,\B))\leq\sup\nolimits_{u\in\B}\{\,|(f,l)(t,x,u)-(f,l)(t,y,u)|\,\}
\end{equation}
for all $t\in[0,\infty)$, $x\in\R^{\scriptscriptstyle N}$. In view of (A1) and \eqref{rwwar-1}, the set-valued map $x\rightsquigarrow (f,l)(t,x,\B)$ is continuous in the sense of the Hausdorff metric for all $t\in[0,\infty)$. Therefore,  (h3) holds.
We observe that (A3) implies
$$\G H^{\ast}(t,x,\cdot)+\{0\}\times[0,\infty)\subset(f,l)(t,x,\B)+\{0\}\times[0,\infty)\subset\E H^{\ast}(t,x,\cdot)+\{0\}\times[0,\infty).$$
Hence, $\{(f(t,x,u),l(t,x,u)+r)\mid u\in \B,\,r\geq 0\}=\E H^{\ast}(t,x,\cdot)$. Therefore, the convexity of the set $\{(f(t,x,u),l(t,x,u)+r)\mid u\in\B,\,r\geq 0\}$ is a consequence of the convexity of $\E H^{\ast}(t,x,\cdot)$, which means that condition (h4) is satisfied. We observe that (h6) follows directly from (A1).
Since $f(t,x,\B)=\D H^{\ast}(t,x,\cdot)$ for all $t\in[0,\infty)$, $x\in\R^{\scriptscriptstyle N}$, the condition (OPC) follows directly from the condition $\tn{(OPC)}_H$.
\end{proof}

\vspace{-4mm}
\pagebreak

\begin{Rem}\label{rem-bledy}
In \cite{B-F-2020,B-F-2022} the author constructs representations that share several structural features with ours and are based on techniques from \cite{AM,AM1}. In the $\text{(h)}''$ setting, such\linebreak representations may fail to satisfy the full set of properties, while the representation\linebreak proposed here does. The reasons are as follows:

The function $\omega$ in Proposition 4.6 (ii) in \cite{B-F-2022} is given by the formula:
\begin{align}
& \omega(t,x):=c(t)(1+|x|)+\varrho(t,x)+|H(t,x,0)|,\;\;\tn{where}\label{rem-defom-1}\\[-1mm]
& \varrho(t,x):=\max\{0,\sup\{H^{\ast}(t,x,v)\mid v\in\D H^{\ast}(t,x,\cdot)\}\}.\label{rem-defom-2}
\end{align}
In Example \ref{Ex3} below, we show that the epigraphical representation $(\B,f,l)$, as used in Proposition 4.6 of \cite{B-F-2022}, has discontinuous functions $f(t,\cdot,u)$ and $l(t,\cdot,u)$. Consequently,\linebreak the condition $\tn{(h)}''\tn{(h6)}$ fails, that is,
$$\sup\nolimits_{u\in\B}\big\{|f(t,x,u)-f(t,y,u)|+|l(t,x,u)-l(t,y,u)|\big\}\leq\, k(t)|x-y|.$$

The role of the auxiliary function $\varrho$ in \cite{B-F-2020} appears to be ambiguous: on page 376 it is defined by \eqref{rem-defom-2}, whereas elsewhere it is interpreted as a boundedness condition for the Legendre-Fenchel conjugate on its effective domain. Either interpretation leads to the same conclusion: the epigraphical representation constructed in \cite{B-F-2020} need not satisfy\linebreak the condition  $\tn{(h)}''\tn{(h5)}$, i.e.,
\begin{equation}\label{rem-inqfl}
\sup_{(x,u)\,\in\,\mathrm{bd}\,A\times\,\B}\big\{|f(t,x,u)|+|l(t,x,u)|\big\}\leq q(t),\;\;\textnormal{where}\;\;q\in\mathscr{L}_{\tn{loc}}.
\end{equation}
The failure of \eqref{rem-inqfl} is independent of the status of $\varrho$ and stems from the factor $c(t)(1+|x|)$ in the definition of $\omega$; see Example~\ref{Ex4}.
\end{Rem}

\begin{Ex}\label{Ex3}
Let us define the Hamiltonian $H:[0,\infty)\times\R\times\R\times\R\rightarrow\R$ by the formula:
\begin{equation*}
 H(t,x,p,q):=\max\{\,|p|\,|x|-q\,e^{-t},0\,\}+|p|.
\end{equation*}
The Legendre-Fenchel conjugate of $H(t,x,\cdot,q)$ with $q>0$ is given by\vspace{2mm}
\begin{align*}
&\;\;H^{\ast}(t,x,v,q)=\left\{
\begin{array}{ccl}
+\infty, & \tn{if} & |v|>|x|+1,\;x\not=0,\\
q\,\max\left\{\frac{\displaystyle |v|-1}{\displaystyle e^{t}\,|x|},\,0\right\}\!\!, & \tn{if} & |v|\leq |x|+1,\;x\not=0,
\\[3mm]
0, & \tn{if} & |v|\leq 1,\; x=0,\\[-0.5mm]
+\infty, & \tn{if} & |v|>1,\;x=0,
\end{array}
\right.\label{ex-th}\\[2mm]
&\;\;\D H^{\ast}(t,x,\cdot,q)=[-|x|-1,\,|x|+1]\;\tn{for all}\; t\in[0,\infty),\, x\in\R,\,q>0.\nonumber
\end{align*}

\noindent It is not difficult to verify that  $H$ defined above, together with  $\Omega:=(-\infty,0]$, satisfies all assumptions H.1-2 and C.1-2 from \cite{B-F-2022}. We use the notation  $H(t,x,p)$ for $H(t,x,p,1)$.  The function $\varrho(t,x)$ given by \eqref{rem-defom-2} with 
$H^{\ast}(t,x,v)=H^{\ast}(t,x,v,1)$ can be easily calculated
\begin{equation*}
\varrho(t,x)
=\left\{
\begin{array}{ccl}
0, & \tn{if} & x=0,\\
e^{-t}\!, & \tn{if} & x\not=0.
\end{array}
\right.
\end{equation*}

We show that the epigraphical representation $(\B,f,l)$ of the Hamiltonian $H(t,x,p,1)$, as used in Proposition 4.6 of \cite{B-F-2022}, has discontinuous functions $f(t,\cdot,u)$ and $l(t,\cdot,u)$. Notably, Proposition 4.6 of \cite{B-F-2022} refers to Theorem 4.1 of \cite{B-F-2020}, where the representation is constructed. The representation in \cite{B-F-2022} is given by \eqref{rep-t-01} and satisfies the following property, which follows from \eqref{rep-t-02}:
\begin{equation}\label{peop-eprrep0}
(f,l)(t,x,u)=\omega(t,x)u\;\;\textnormal{for all}\;\;\omega(t,x)u\in\E H^{\ast}(t,x,\cdot,1),
\end{equation}
where $\omega(t,x)$ is defined by \eqref{rem-defom-1}, that is, 
$\omega(t,x)=c(t)(1+|x|)+\varrho(t,x)+|H(t,x,0,1)|$.\linebreak Without loss of generality, we can assume that $c(t)\equiv1$. We observe that $H(t,x,0,1)=0$. Thus, $\omega(t,x)=1+|x|+\varrho(t,x)$. So, $\omega(t,x)\leq 3$ for all $t\in[0,\infty)$, $x\in[-1,1]$.  In consequence, $\omega(t,x)/3\in[-1,1]$ for all $t\in[0,\infty)$, $x\in[-1,1]$. Let $u:=(1/3,1/3)$. Since $H^*(t,x,p,1)=0$ for all $t\in[0,\infty)$, $x\in\R$, $p\in[-1,1]$, we have
$\omega(t,x)u\in\E H^{\ast}(t,x,\cdot,1)$ for all $t\in[0,\infty)$ and $x\in[-1,1]$. In view of \eqref{peop-eprrep0}, for $u=(1/3,1/3)$ we obtain
$$(f,l)(t,x,u)=\omega(t,x)u\;\;\textnormal{for all}\;\;t\in[0,\infty),\;x\in[-1,1].$$
Therefore, for $u=(1/3,1/3)$ we have
\begin{equation*}
f(t,x,u)=l(t,x,u)=
\left\{
\begin{array}{lcl}
\frac{1}{3}, & \textnormal{if} & x=0,\\[2mm]
\frac{1}{3}+\frac{1}{3}|x|+\frac{1}{3}e^{-t}\!\!, & \textnormal{if} & x\in[-1,0)\cup(0,1].
\end{array}
\right.
\end{equation*}
So, the functions $f(t,\cdot,u)$ and $l(t,\cdot,u)$ are discontinuous.
\end{Ex}

\begin{Ex}\label{Ex4}
Let us define the Hamiltonian $H:[0,\infty)\times\R^2\times\R^2\to\R$ by the formula:
$$H(t,x_1,x_2,p_1,p_2)=|(p_1,p_2)|.$$
The Legendre-Fenchel conjugate of $H(t,x_1,x_2,\cdot,\cdot)$ is given by
\begin{equation*}
H^*(t,x_1,x_2,v_1,v_2)
=\left\{
\begin{array}{ccl}
0, & \textnormal{if} & |(v_1,v_2)|\leq 1,\\
+\infty, & \textnormal{if} & \textnormal{otherwise}.
\end{array}
\right.
\end{equation*}
It is not difficult to verify that $H$ defined above, together with  $\Omega:=(-\infty,0]\times\R$, satisfies all assumptions H.2.1-6 and OPC in \cite{B-F-2020}. In the present paper, we denote the set $\Omega$ by  $A$.

We show that the epigraphical representation $(\B,\!f\!,l)$ of this Hamiltonian $H$, constructed as in Theorem 4.1 of \cite{B-F-2020}, does not satisfy \eqref{rem-inqfl}. The representation in \cite{B-F-2020} is given by \eqref{rep-t-01} and satisfies the following property, which follows from \eqref{rep-t-02}:
\begin{equation}\label{peop-eprrep00}
(f,l)(t,x,u)=\omega(t,x)u\;\;\textnormal{for all}\;\;\omega(t,x)u\in \E H^{\ast}(t,x,\cdot),
\end{equation}
Let $(u_1,u_2,u_3):=(0,0,1)$. Then $\omega(t,x_1,x_2)(u_1,u_2,u_3)\in \E H^{\ast}(t,x_1,x_2,\cdot,\cdot)$. By \eqref{peop-eprrep00}, 
\begin{equation}\label{wzor}
(f,l)(t,x_1,x_2,0,0,1)=(0,0,\omega(t,x_1,x_2))\;\;\textnormal{for all}\;\;t\in[0,\infty),\;x_1,x_2\in\R.
\end{equation}
Of course, $H(t,x_1,x_2,0,0)=0$.  Without loss of generality, we can assume that $c(t)\equiv1$. However, for the function $\varrho(t,x_1,x_2)$ we only require that it is nonnegative. Then, the function $\omega(t,x_1,x_2)$ defined as in \eqref{rem-defom-1} takes the form
\begin{equation}\label{wzor-gamma}
\omega(t,x_1,x_2)=1+|(x_1,x_2)|+\varrho(t,x_1,x_2).
\end{equation}
We assume, by contradiction, that the functions $f$ and $l$ constructed above, together with the given set  $A=\Omega$, satisfy condition \eqref{rem-inqfl}. Therefore, by \eqref{wzor} and \eqref{wzor-gamma}, we obtain
\begin{equation}\label{ineq-1}
1+|(x_1,x_2)|+\varrho(t,x_1,x_2)\;\leq\; q(t)\;\;\textnormal{for all}\;\;(x_1,x_2)\in\mathrm{bd}\,A,\;t\in[0,\infty).
\end{equation}
We observe that $\mathrm{bd}\,A=\{0\}\times\R$. Let $(x_1,x_2):=(0,i)\in\mathrm{bd}\,A$, where $i\in\N$. Then, by \eqref{ineq-1}, we have $i\leq q(t)$ for all $i\in\N$ and $t\in[0,\infty)$. Passing to the limit as  $i\to\infty$ we obtain the contradiction $q\equiv+\infty$. Consequently, the functions $f$ and $l$ constructed above, together with the given set  $A=\Omega$, do not satisfy  \eqref{rem-inqfl}. As can be easily observed, this is due to the factor $c(t)(1+|x|)$ present in the definition of $\omega$ (see \eqref{rem-defom-1}).
\end{Ex}

Now we present an example consisting of a Hamiltonian $H$ and a set $A$ that satisfy $\text{(h)}''$ and $\tn{(OPC)}_H$. Moreover, we provide an explicit epigraphical representation and show that no regular graphical representation exists. This epigraphical representation, in turn, allows us to demonstrate in the next section (see Remark \ref{rmwb}) that the condition\linebreak $\tn{(B)}_H$ -- defined there -- is a weaker version of condition (B) in \cite{B-F-2019}.

\begin{Ex}\label{Ex1}
Let us define the Hamiltonian $H:[0,\infty)\times\R\times\R\rightarrow\R$ by the formula:
\begin{equation*}
H(t,x,p):=\max\{\,\alpha(t)\,|p|\,|x|-\alpha(t)\,e^{-\gamma t},0\,\}+|p|,
\end{equation*}
where $\alpha(t)\in L^\infty([0,\infty);\R^+)$ and $\gamma\in\R^+$.  Moreover, we have
\begin{align*}
& H^{\ast}(t,x,v)=\left\{
\begin{array}{ccl}
+\infty, & \tn{if} & |v|>\alpha(t)\,|x|+1,\;x\not=0,\\
\max\left\{\frac{\displaystyle |v|-1}{\displaystyle e^{\gamma t}\,|x|},\,0\right\}\!\!, & \tn{if} & |v|\leq \alpha(t)\,|x|+1,\;x\not=0, \\
0, & \tn{if} & |v|\leq 1,\; x=0,\\[-1mm]
+\infty, & \tn{if} & |v|>1,\;x=0,
\end{array}
\right.\\
& \D H^{\ast}(t,x,\cdot)=[-\alpha(t)\,|x|-1,\,\alpha(t)\,|x|+1]\;\tn{for all}\; t\in[0,\infty),\, x\in\R.\\[-6mm]
\end{align*}
 This Hamiltonian satisfies $\tn{(h)}_H''$ and $\tn{(OPC)}_H$  with $A\!=\!(-\infty,0]$ and $\lambda(t,x)\!=\!\alpha(t)\,|x|\!+\!\alpha(t)\!+\!1$.

To obtain a representation $(U,f,l)$ of the Hamiltonian we set: $U\equiv[-1,1]$ and $$f:[0,\infty)\times\R\times\R\to\R,\quad l:[0,\infty)\times\R\times\R\to\R$$  are given by \vspace{-3mm}
\begin{align*}
 f(t,x,u) &=\left\{
\begin{array}{ccl}
-\alpha(t)\,|x|-1, & \tn{if} & u\in(-\infty,-1], \\
\alpha(t)\,|x|\,(2u+1)-1, & \tn{if} & u\in[-1,-1/2], \\
2\,u, & \tn{if} & u\in[-1/2,1/2],\\
\alpha(t)\,|x|\,(2u-1)+1, & \tn{if} & u\in[1/2,1],\\
\alpha(t)\,|x|+1, & \tn{if} & u\in[1,\infty),
\end{array}
\right.\\[0mm]
l(t,x,u) &=\left\{
\begin{array}{ccl}
\alpha(t)\,e^{-\gamma t}, & \tn{if} & u\in(-\infty,-1], \\
-\alpha(t)\,(2u+1)\,e^{-\gamma t}, & \tn{if} & u\in[-1,-1/2], \\
0, & \tn{if} & u\in[-1/2,1/2],\\
\alpha(t)\,(2u-1)\,e^{-\gamma t}, & \tn{if} & u\in[1/2,1],\\
\alpha(t)\,e^{-\gamma t}, & \tn{if} & u\in[1,\infty).
\end{array}
\right.
\end{align*}

\noindent  The triple $([-1,1],f,l)$ is an epigraphical representation of  $H$
but it is not a graphical\linebreak representation. Indeed, a graphical representation  for the Hamiltonian  in this example with continuous functions $(x,u)\to f(t,x,u)$ and $(x,u)\to l(t,x,u)$ does not exist. To see this let us assume, by a contradiction, that such a representation   $(\B,f,l)$ exists. Let $x_n\!=\!1/n$, $v_n\!=\!(\alpha(t)/n)\!+\!1$. Observe that $H^{\ast}(t,x_n,v_n)=\alpha(t)\,e^{-\gamma t}\!\!$. Hence $(v_n,\alpha(t)\,e^{-\gamma t})\in\G H^{\ast}(t,x_n,\cdot)$. Since $(\B,f,l)$ is a graphical representation, there exists $u_n\in\B$ such that $f(t,x_n,u_n)=v_n$ and $l(t,x_n,u_n)=\alpha(t)\,e^{-\gamma t}$. Since the set $\B$ is compact, there exists a subsequence (which we do not relabel) $\{u_n\}$ convergent to  $u\in\B$. Passing to the limit as $n\to\infty$ we have $f(t,0,u)\!=\!1$ and $l(t,0,u)\!=\!\alpha(t)\,e^{-\gamma t}\!\!$. Thus,  we get $(1,\alpha(t)\,e^{-\gamma t})\!\in\!\G H^{\ast}(t,0,\cdot)\!=\![-1,1]\!\times\!\{0\}$.\linebreak Therefore, $\alpha(t)\,e^{-\gamma t}=0$, in contradiction to  $\alpha(t)\,e^{-\gamma t}>0$.
\end{Ex}

\section{Hamilton-Jacobi-Bellman equation}
\noindent In the section we investigate the issues of existence, uniqueness and representation for   \eqref{eqhjb}. The problem has been solved  in by Basco-Frankowska (see Thorem 3.3 in \cite{B-F-2019}). But the result of Basco-Frankowska has been obtained for  H-J-B equations related to control problem given by a triple $(U,f,l)$. The assumptions of Theorem 3.3 in \cite{B-F-2019} has been formulated as properties of the dynamics $f$ and the  running cost $l$. In Theorem \ref{eaus-h} we propose an analog to the Basco-Frankowska result but we formulate all assumptions as properties of the Hamiltonian $H(t,x,p)$. The base entity of our consideration is the Hamiltonian $H(t,x,p)$ and its representation given by a triple $(U,f,l)$ is a secondary object.

We propose the following definition of a weak solution to \eqref{eqhjb_0}:
\begin{Def}\label{dws1}
 A lower semicontinuous function $W:[0,\infty)\times A\to\R\cup\{+\infty\}$ is called a weak solution of  \eqref{eqhjb_0}  on  $(0,\infty)\times A$ if there exists a set $C\subset(0,\infty)$ of full measure such that for all $(t,x)\in\D W\cap(C\times\mathrm{bd}\,A)$ one has
\begin{equation*}
\begin{array}{rl}
-p_t+H(t,x,-p_x)\geq 0,&\;\;\forall\,(p_t,p_x)\in \partial W(t,x),\\[1mm]
-p_t+\sup\limits_{v\,\in\,\D H^{\ast}(t,x,\cdot)}\langle v,-p_x\rangle\geq 0,&\;\;\forall\,(p_t,p_x,0)\in N_{\E W}(t,x,W(t,x)),
\end{array}
\end{equation*}
and for all $(t,x)\in\D W\cap(C\times\mathrm{int}\,A)$ one has
\begin{equation*}
\begin{array}{rl}
-p_t+H(t,x,-p_x)=0,&\;\;\forall\,(p_t,p_x)\in \partial W(t,x),\\[1mm]
-p_t+\sup\limits_{v\,\in\,\D H^{\ast}(t,x,\cdot)}\langle v,-p_x\rangle=0,&\;\;\forall\,(p_t,p_x,0)\in N_{\E W}(t,x,W(t,x)).
\end{array}
\end{equation*}
\end{Def}
Using methods presented in Section \ref{sec-rch} we can find a triple $(U,f,l)$ that is a representation (or epigraphical representation) of the given Hamiltonian $H(t,x,p)$. Then we can define the augmented Hamiltonian $\bar{H}:[0,\infty)\times \R^{\scriptscriptstyle N}\times \R^{\scriptscriptstyle N}\times \R\to \R$ by
$$\bar{H}(t,x,p,q)=\sup\nolimits_{\,u\,\in\, U(t)}\{\,\langle\, p,f(t,x,u)\,\rangle-q\,l(t,x,u)\,\}.$$ We recall the definition of a weak solution to the H-J-B equation used in \cite{B-F-2019}:

\begin{Def}\label{dws2}
A lower semicontinuous function $W:[0,\infty)\times A\to\R\cup\{+\infty\}$ is called a weak solution of  \eqref{eqhjb_0}  on  $(0,\infty)\times A$ if there exists a set $C\subset(0,\infty)$ of full measure such that for all $(t,x)\in\D W\cap(C\times\mathrm{bd}\,A)$ one has
$$-p_t+\bar{H}(t,x,-p_x,-q)\geq 0,\;\;\forall\,(p_t,p_x,q)\in N_{\E W}(t,x,W(t,x)),$$
and for all $(t,x)\in\D W\cap(C\times\mathrm{int}\,A)$ one has
$$-p_t+\bar{H}(t,x,-p_x,-q)=0,\;\;\forall\,(p_t,p_x,q)\in N_{\E W}(t,x,W(t,x)).$$
\end{Def}
Definitions \ref{dws1} and  \ref{dws2} are equivalent. This is a consequence of  classical properties of nonsmooth analysis tools.  First, the regular normal cone property states that $q \leq 0$ for all $(p_t, p_x, q) \in N_{\E W}(t, x, W(t, x))$. Second, there is a direct connection between the regular normal cone and the subdifferential: $(p_t, p_x, -1) \in N_{\E W}(t, x, W(t, x))$ if and only if $(p_t, p_x) \in \partial W(t, x)$. Lastly, the equation $\D H^*(t, x, \cdot) = f(t, x, U(t))$ is confirmed by Lemma~\ref{rep-rde}. By this arguments we obtain that:

\begin{Prop}\label{eqw-defsol}
Let a triple $(\,U,f,l\,)$ be a representation of $H$  satisfying $\tn{(h)}''$. Then $W$ is a weak solution of  \eqref{eqhjb_0}  in the sense of Definition \ref{dws1} if and only if  $W$ is a weak solution of  \eqref{eqhjb_0}  in the sense of Definition \ref{dws2}.
\end{Prop}

\vspace{-1mm}
We recall the main result of \cite{B-F-2019} that states the uniqueness and existence of a weak solution to the H-J-B equation subject to the vanishing at infinity condition \eqref{eqhjb}.

\vspace{-1mm}
\begin{Th}[\tn{\cite[Thm. 3.3]{B-F-2019}}]\label{eaus-fl}
Assume that a triple $(U,f,l)$ satisfies $\tn{(h)}''\!\!$ and $\tn{(OPC)}$. Let $W:[0,\infty)\times A\to\R\cup\{+\infty\}$ be a lower semicontinuous function such that $\D \mathcal{V}(t,\cdot)\subset\D W(t,\cdot)\neq\emptyset$  for all large $t>0$ and
\begin{equation}\label{ic-hjb-2}
\lim\nolimits_{\,t\to\infty}\sup\nolimits_{\,x\,\in\,\D W(t,\,\cdot\,)}|W(t,x)|=0.
\end{equation}
Then the following statements are equivalent:
\begin{enumerate}
    \item[\tn{\bf{(i)}}] $W=\mathcal{V}$, where $\mathcal{V}$ is the value function given by \eqref{def-vf2}\tn{;}
    \item[\tn{\bf{(ii)}}] $W$ is a weak solution of  \eqref{eqhjb} , and  $t\rightsquigarrow\E W(t,\cdot)$ is locally absolutely continuous.
\end{enumerate}
Moreover, if in addition the following condition
\begin{enumerate}
\item[$\pmb{\tn{(B):}}$] $\D\mathcal{V}\not=\emptyset$ and there exist $T>0$ and $\psi\in L^1([T,\infty);[0,\infty))$ such that for all\linebreak $(t_0,x_0)\in \D \mathcal{V}\cap[T,\infty)\times\R^{\scriptscriptstyle N}$ and any $(x,u)(\cdot)\in S_{\!\!f\,}(t_0,x_0)$  one has $|l(t,x(t),u(t))|\leq\psi(t)$ for almost all $t\in[t_0,\infty)$
\end{enumerate}
holds true, then  $\mathcal{V}$ is the unique weak solution of \eqref{eqhjb}  satisfying \eqref{ic-hjb-2} with locally absolutely continuous $t\rightsquigarrow\E\mathcal{V}(t,\cdot)$.
\end{Th}

\vspace{-1mm}
The solution of the H-J-B equation may be represented by the value function of a corresponding calculus of variation problem that is defined in the following Proposition.

\vspace{-1mm}
\begin{Prop}\label{pfv-cvp}
Assume that a Hamiltonian $H$ satisfies $\tn{(h)}''_H$. Then the value function $V:[0,\infty)\times A\to\R\cup\{+\infty\}$  defined for the calculus of variations problem by the formula
\begin{equation}\label{def-vf1}
V(t_0,x_0)=\inf_{x(\cdot)\,\in\, S_{\!\!H}(t_0,x_0)}\int_{t_0}^{\infty}H^{\ast}(t,x(t),\dot{x}(t))\,dt,
\end{equation}
where $S_{\!\!H}(t_0,x_0)$ denotes the set of all trajectories such that
\begin{equation}\label{def-vf1-svs}
\left\{\begin{array}{l}
\dot{x}(t)\in\D H^{\ast}(t,x(t),\cdot),\;\;\tn{a.e.}
\;\;t\in[t_0,\infty), \\
x(t_0)=x_0,\quad x([t_0,\infty))\subset A,
\end{array}\right.
\end{equation}
is well-defined, lower semicontinuous, and for every $(t_0,x_0)\in \D V$ there exists an\linebreak optimal trajectory $\bar{x}(\cdot)$ of $V$ at $(t_0,x_0)$. Moreover, if in addition the following condition
\begin{enumerate}[leftmargin=16mm]
\item[$\pmb{\tn{(B)}_H:}$] $\D V\not=\emptyset$ and there exist $T>0$ and $\psi\in L^1([T,\infty);[0,\infty))$ such that for all $(t_0,x_0)\in \D V\cap[T,\infty)\times\R^{\scriptscriptstyle N}$ and any $x(\cdot)\in S_{\!\!H}(t_0,x_0)$  one has $|H^{\ast}(t,x(t),\dot{x}(t))|\leq\psi(t)$ for almost all $t\in[t_0,\infty)$
\end{enumerate}
holds true, then \vspace{-1mm}
\begin{equation}\label{pfv-bc}
\lim\nolimits_{\,t\to\infty}\sup\nolimits_{\,x\,\in\,\D V(t,\,\cdot\,)}|V(t,x)|=0.
\end{equation}
\end{Prop}

\begin{proof}
Combining our assumptions with the results of \cite[Chap. 14]{R-W}, we deduce that $V$ is well-defined and satisfies $-\infty<V(t_0,x_0)\leq +\infty$ for all $t_0\in[0,\infty)$ and $x_0\in A$.

The lower semicontinuity of the value function $V$ follows directly from Proposition \ref{glscfv},\linebreak stated in the next section, with $H_n:=H$. 

From Lemma \ref{wcsh}, also stated in the next section (again with $H_n:=H$), it is not difficult to deduce that for every $(t_0,x_0)\in \D V$ there exists $\bar{x}(\cdot)\in S_{\!\!H}(t_0,x_0)$ such that $V(t_0,x_0)\;=\;\int_{t_0}^{\infty}H^{\ast}(t,\bar{x}(t),\dot{\bar{x}}(t))\,dt$, which means that $\bar{x}(\cdot)$ is an optimal trajectory of  $V$ at $(t_0,x_0)$.

It remains to prove \eqref{pfv-bc}. In view of $\tn{(B)}_H$ there exists $(s_0,y_0)$ belonging to $\D V$. Therefore, there exists an optimal trajectory $\bar{x}(\cdot)$ of $V$ at $(s_0,y_0)$. By the definition of $V$,
$$V(t,\bar{x}(t))\;\leq\;\int_{t}^{\infty}H^{\ast}(s,\bar{x}(s),\dot{\bar{x}}(s))\,d\!s\;<\;+\infty$$
for all $t\in[s_0,\infty)$. Therefore, we have $\D V(t,\cdot)\neq\emptyset$ for all $t\in[s_0,\infty)$. This means that  $\sup\nolimits_{\,x\,\in\,\D V(t,\,\cdot\,)}|V(t,x)|$ is well-defined  for all $t\in[s_0,\infty)$. For every $(t,x)\in \D V$ we choose an optimal trajectory  $\bar{z}_{t,x}(\cdot)$ of $V$ at $(t,x)$. The latter, together with $\tn{(B)}_H$,  implies that
$$|V(t,x)|\;\leq\;\int_{t}^{\infty}|H^{\ast}(s,\bar{z}_{t,x}(s),\dot{\bar{z}}_{t,x}(s))|\,d\!s\;\leq\;\int_{t}^{\infty}\psi(s)\,d\!s,$$
for all $t\geq \max\{s_0,T\}$, $x\in\D V(t,\cdot)$. Hence, we deduce that $V$ satisfies \eqref{pfv-bc}.
\end{proof}

\begin{Th}\label{eqw-fws}
Assume that $A$ is a nonempty closed subset of $\R^{\scriptscriptstyle N}$ and that  the set-valued\linebreak map $U:[0,\infty)\rightsquigarrow\R^{\scriptscriptstyle M}$ is  measurable  with nonempty closed images. Additionally, let\linebreak
$H:[0,\infty)\times\R^{\scriptscriptstyle N}\times\R^{\scriptscriptstyle N}\to\R$ satisfy \tn{(H1)} and \tn{(H2)}, and let  $f:[0,\infty)\times\R^{\scriptscriptstyle N}\times\R^{\scriptscriptstyle M}\to\R^{\scriptscriptstyle N}$ and\linebreak $l:[0,\infty)\times\R^{\scriptscriptstyle N}\times\R^{\scriptscriptstyle M}\to\R$ satisfy \tn{(h1)}. If  $(U,f,l)$ is the epigraphical representation of $H$,
then the value functions $V$ and $\mathcal{V}$ are well-defined, and  $$V(t_0,x_0)=\mathcal{V}(t_0,x_0) \;\;\;\it{for all}\;\;\;(t_0,x_0)\in[0,\infty)\times A.$$
If $\bar{x}(\cdot)$ is an optimal trajectory of $V$ at  $(t_0,x_0)\in\D V$, then there exists a measurable function $\bar{u}(\cdot)$ such that $(\bar{x},\bar{u})(\cdot)$  is the optimal pair of $\mathcal{V}$ at  $(t_0,x_0)$. Conversely, if $(\bar{x},\bar{u})(\cdot)$ is an optimal pair of $\mathcal{V}$ at $(t_0,x_0)\in\D\mathcal{V}$, then $\bar{x}(\cdot)$ is the optimal trajectory of $V$ at  $(t_0,x_0)$.
\end{Th}

\begin{proof}
Under our assumptions, we have $-\infty<V(t_0,x_0)\leq +\infty$ and $-\infty<\mathcal{V}(t_0,x_0)\leq +\infty$ for all $t_0\in[0,\infty)$ and $x_0\in A$.

Fix $(t_0,x_0)\in[0,\infty)\times A$. We first show that $\mathcal{V}(t_0,x_0)\geq V(t_0,x_0)$. If $\mathcal{V}(t_0,x_0)=+\infty$, then $\mathcal{V}(t_0,x_0)\geq V(t_0,x_0)$. Suppose next that $(t_0,x_0)\in\D\mathcal{V}$. Fix $\varepsilon>0$. By the definition of $\mathcal{V}$, there exists  $(x,u)(\cdot)\in S_{\!\!f\,}(t_0,x_0)$ such that
\begin{equation}\label{eqw-fws-1}
\mathcal{V}(t_0,x_0)+\varepsilon\;\geq\;\int_{t_0}^{\infty}l(t,x(t),u(t))\,dt.
\end{equation}
Since $(f,l)(t,x,U(t))\subset\E H^{\ast}(t,x,\cdot)$ for all $t\in[0,\infty)$, $x\in\R^{\scriptscriptstyle N}$, we deduce  $l(t,x(t),u(t))\geq H^{\ast}(t,x(t),f(t,x(t),u(t)))$ for a.e. $t\in[t_0,\infty)$. In particular, $f(t,x(t),u(t))\in\D H^{\ast}(t,x(t),\cdot)$ for a.e. $t\in[t_0,\infty)$. So, $\dot{x}(t)\in\D H^{\ast}(t,x(t),\cdot)$ for a.e. $t\in[t_0,\infty)$. Moreover, $x([t_0,\infty))\subset A$. Along with the inequality \eqref{eqw-fws-1}, this implies that
\begin{eqnarray*}\label{eqw-fws-2}
\mathcal{V}(t_0,x_0)+\varepsilon &\geq & \int_{t_0}^{\infty}l(t,x(t),u(t))\,dt\\
& \geq & \int_{t_0}^{\infty}H^{\ast}(t,x(t),f(t,x(t),u(t)))\,dt\\
& = & \int_{t_0}^{\infty}H^{\ast}(t,x(t),\dot{x}(t))\,dt\;\;\geq\;\; V(t_0,x_0).
\end{eqnarray*}
 As $\varepsilon>0$ can be arbitrary small, the latter inequality implies $\mathcal{V}(t_0,x_0)\geq V(t_0,x_0)$.

We show next that $V(t_0,x_0)\geq\mathcal{V}(t_0,x_0)$. If $V(t_0,x_0)=+\infty$, then $V(t_0,x_0)\geq\mathcal{V}(t_0,x_0)$.\linebreak Let us assume that $(t_0,x_0)\in\D V$.
Fix $\varepsilon>0$. In view of the definition of $V$, there exists $x(\cdot)\in S_{\!\!H}(t_0,x_0)$ such that
\begin{equation}\label{eqw-fws-3}
V(t_0,x_0)+\varepsilon\;\geq\;\int_{t_0}^{\infty}H^{\ast}(t,x(t),\dot{x}(t))\,dt.
\end{equation}
Thus, $\dot{x}(t)\in\D H^{\ast}(t,x(t),\cdot)$ for a.e. $t\in[t_0,\infty)$. So, $(\dot{x}(t),H^{\ast}(t,x(t),\dot{x}(t)))\in\G H^{\ast}(t,x(t),\cdot)$ for a.e. $t\in[t_0,\infty)$.
Since $\G H^{\ast}(t,x,\cdot)\subset(f,l)(t,x,U(t))$ for all $t\in[t_0,\infty)$, $x\in\R^{\scriptscriptstyle N}$, by\linebreak \cite[Thm. 8.2.10]{A-F} there exists a  measurable function $u(\cdot)$ such that $(\dot{x}(t),H^{\ast}(t,x(t),\dot{x}(t)))=(f,l)(t,x(t),u(t))$ and $u(t)\!\in\! U(t)$ for a.e. $t\!\in\![t_0,\infty)$. So, $(x,u)(\cdot)\!\in\! S_{\!\!f\,}(t_0,x_0)$. Thus, by \eqref{eqw-fws-3},
$$V(t_0,x_0)+\varepsilon\;\geq\;\int_{t_0}^{\infty}H^{\ast}(t,x(t),\dot{x}(t))\,dt\;=\;\int_{t_0}^{\infty}l(t,x(t),u(t))\,dt\;\geq\;\mathcal{V}(t_0,x_0),$$
 which implies $ V(t_0,x_0)\geq \mathcal{V}(t_0,x_0)$.

From the equality $V(t_0, x_0) = \mathcal{V}(t_0, x_0)$ and its proof it follows the second part of this theorem's statement.
\end{proof}

The main result of the section is a reformulation of Basco-Frankowska result recalled in Theorem \ref{eaus-fl}. We replaced the assumptions formulated as properties of the triple $(U,f,l)$\linebreak by assumptions formulated by using only the Hamiltonian $H(t,x,p)$ and its Fenchel\linebreak conjugate $H^*(t,x,v)$.

\begin{Th}\label{eaus-h}
Assume that $H$ satisfies $\tn{(h)}_H''$ and  $\tn{(OPC)}_H$. Let $W:[0,\infty)\times A\to\R\cup\{+\infty\}$ be a lower semicontinuous function such that $\D V(t,\cdot)\!\subset\!\D W(t,\cdot)\neq\emptyset$  for all large $t>0$ and the final condition \eqref{ic-hjb-2} is satisfied. Then the following statements are equivalent:
\begin{enumerate}
    \item[\tn{\bf{(i)}}] $W=V$, where $V$ is the value function given by \eqref{def-vf1}\tn{;}
    \item[\tn{\bf{(ii)}}] $W$ is a weak solution of  \eqref{eqhjb_0} on the set $(0,\infty)\times A$ , and  $t\rightsquigarrow\E W(t,\cdot)$ is locally absolutely continuous.
\end{enumerate}
Moreover, if in addition $\tn{(B)}_H$ holds true, then $V$ is the unique weak solution of \eqref{eqhjb}  with locally absolutely continuous $t\rightsquigarrow\E V(t,\cdot)$.
\end{Th}
\begin{proof}
In view of Theorem \ref{rwwar} there exists a representation $(U,f,l)$ of $H$ satisfying $\tn{(h)}''$ and  $\tn{(OPC)}$. Moreover, in view of Theorem \ref{rwwar-i}, the triple $(U,f,l)$ is also an epigraphical representation of $H$. By Theorem \ref{eqw-fws}, we have $V=\mathcal{V}$, where the value function $\mathcal{V}$ is given by \eqref{def-vf2}. So, the equivalence of (i) and (ii) is a consequence of Theorem \ref{eaus-fl}. By Proposition \ref{pfv-cvp}, the value function $V$ satisfies the final condition \eqref{pfv-bc}. So, it is the unique weak solution to
\eqref{eqhjb}.
\end{proof}

\begin{Rem}\label{rmwb}
Since the condition $\tn{(B)}_H$ is weaker than condition (B), one cannot infer the existence of a solution from Theorem \ref{eaus-h} based on the existence of a solution from Theorem \ref{eaus-fl}.

Assume that $(U,f,l)$ is a representation of $H$  satisfying $\tn{(h)}''$. If (B) is satisfied, then $\tn{(B)}_H$ is also satisfied. Indeed,
let $\dot{x}(t)\in\D H^{\ast}(t,x,\cdot)=f(t,x,U(t))$ for almost every\linebreak $t\in[0,\infty)$. Thus, by \cite[Thm. 8.2.10]{A-F}, there exists a measurable function $u(\cdot)$ such that $\dot{x}(t)=f(t,x(t),u(t))$ and $u(t)\in U(t)$ for almost every $t\in[0,\infty)$. By \eqref{rep-inc} we obtain
\begin{equation}\label{eqwcb}
\phi(t)\leq H^{\ast}(t,x(t),\dot{x}(t)) = H^{\ast}(t,x(t),f(t,x(t),u(t)))\leq l(t,x(t),u(t))\leq\psi(t)
\end{equation}
for almost every $t\in[0,\infty)$. Thus,  $|H^{\ast}(t,x(t),\dot{x}(t))|\leq \psi(t)+|\phi(t)|$ for almost every $t\in[0,\infty)$. This means that  $\tn{(B)}_H$ holds, provided that (B) is satisfied.

The converse of the above implication is not true. Indeed, we consider the Hamiltonian $H$ as in Example \ref{Ex1}. Then $H$  satisfies $\tn{(h)}_H''$ and $\tn{(OPC)}_H$ with $A=(-\infty,0]$  and  $\lambda(t,x)=\alpha(x)\,|x|+\alpha(t)+1$. Moreover, $V(\cdot,\cdot)\equiv 0$ and $\tn{(B)}_H$ holds with $\psi(t)=\alpha(t)\,e^{-\gamma t}$. We consider $f$ and $l$ as in Example \ref{Ex1}. Then $([-1,1],f,l)$ is a representation of $H$ with\linebreak $f$ and $l$ satisfying $\tn{(h)}''$ and (OPC)  with $A=(-\infty,0]$. Moreover, $\mathcal{V}(\cdot,\cdot)\equiv 0$ and (B) holds with $\psi(t)=\alpha(t)\,e^{-\gamma t}$. However, if we modify the functions $f$ and $l$ as follows
\begin{align*}
& \hat{f}(t,x,u_1,u_2):=f(t,x,u_1)\;\;\tn{for all}\;\;t\in[0,\infty),\; x,u_1,u_2\in\R,\\
& \hat{l}(t,x,u_1,u_2):=l(t,x,u_1)+\min\{\,|u_2|,1\}\;\;\tn{for all}\;\;t\in[0,\infty),\; x,u_1,u_2\in\R,
\end{align*}
then $([-1,1]\times[-1,1],\hat{f},\hat{l}\,)$ is a representation of $H$ with $\hat{f}$ and $\hat{l}$ satisfying $\tn{(h)}''$ and (OPC)  with $A=(-\infty,0]$. Moreover, $\mathcal{V}(\cdot,\cdot)\equiv 0$, but (B) does not hold. Indeed, we set $(x,u_1,u_2)(\cdot)\equiv(x_0,0,1)$ on $[t_0,\infty)$, where $(t_0,x_0)\in[0,\infty)\times A$. Then $(x,u_1,u_2)(\cdot)\in S_{\!\!\hat{f}\,}(t_0,x_0)$  and $|\hat{l}(t,x(t),u_1(t),u_2(t))|=l(t,x_0,0)+1=1$ for all $t\in[t_0,\infty)$ and $(t_0,x_0)\in[0,\infty)\times A$.\linebreak Because the functions $t\to|\hat{l}(t,x(t),u_1(t),u_2(t))|$ do not belong to $L^1([t_0,\infty);[0,\infty))$ for all $t_0\geq 0$ and all $x_0\in A$, the condition (B) is not satisfied.

However, it turns out that conditions $\tn{(B)}$ and $\tn{(B)}_{\!H}$ can be weakened to be equivalent for the epigraphical representation:

\begin{enumerate}[leftmargin=12.5mm]
\item[$\pmb{\tn{(B)}^+:}$] $\D\mathcal{V}\not=\emptyset$ and there exist $T>0$ and $\psi\in L^1([T,\infty);[0,\infty))$ such that for all $(t_0,x_0)\in \D \mathcal{V}\cap[T,\infty)\times\R^{\scriptscriptstyle N}$ an optimal pair $(x,u)(\cdot)\in S_{\!\!f}(t_0,x_0)$ exists with  $|l(t,x(t),u(t)|\leq\psi(t)$ for almost all $t\in[t_0,\infty)$.\vspace{2mm}
\item[$\pmb{\tn{(B)}_H^+:}$] $\D V\not=\emptyset$ and there exist $T>0$ and $\psi\in L^1([T,\infty);[0,\infty))$ such that for all\\ $(t_0,x_0)\in \D V\cap[T,\infty)\times\R^{\scriptscriptstyle N}$ an optimal trajectory $x(\cdot)\in S_{\!\!H}(t_0,x_0)$ exists\\ with $|H^{\ast}(t,x(t),\dot{x}(t))|\!\leq\!\psi(t)$ for almost all $t\in[t_0,\infty)$.
\end{enumerate}

\vspace{1mm}
\noindent The proof of Proposition \ref{pfv-cvp} remains valid with the weakened condition $\tn{(B)}_{\!H}^+$, thus the value function $V$ still vanishes at infinity, meaning it satisfies \eqref{pfv-bc}.

Now, we show that $\tn{(B)}^+\Rightarrow\tn{(B)}_{\!H}^+$. Let $(x,u)(\cdot)\in S_{\!\!f}(t_0,x_0)$ be the optimal trajectory for the value function $\mathcal{V}$ at $(t_0,x_0)$ as specified by $\tn{(B)}^+$. In view of Theorem \ref{eqw-fws},  $x(\cdot)$ is the optimal trajectory of the value function $V$ at $(t_0,x_0)$. Similarly as above, we show that the pair  $(x,u)(\cdot)$ satisfies inequality \eqref{eqwcb}. From this inequality, we obtain  $|H^{\ast}(t,x(t),\dot{x}(t))|\leq \psi(t)+|\phi(t)|$. Consequently, the chosen trajectory $x(\cdot)$ satisfies $\tn{(B)}_{\!H}^+$.

Now, we show that $\tn{(B)}_{\!H}^+\Rightarrow\tn{(B)}^+$. Let $x(\cdot)\!\in\! S_{\!\!H}(t_0,x_0)$ be the optimal trajectory for the value function $V$ at $(t_0,x_0)$ as specified by $\tn{(B)}_{\!H}^+$. Since the triple $(U,f,l)$ is the epigraphical representation of the Hamiltonian $H$, it follows that
$$(\dot{x}(t),H^{\ast}(t,x(t),\dot{x}(t)))\in\G H^{\ast}(t,x(t),\cdot)\subset(f,l)(t,x,U(t)).$$
Therefore, in view of  \cite[Thm. 8.2.10]{A-F}, there exists a measurable function $u(\cdot)$ such that $(\dot{x}(t),H^{\ast}(t,x(t),\dot{x}(t)))=(f,l)(t,x(t),u(t))$ with $u(t)\in U(t)$. In consequence, we obtain that $H^{\ast}(t,x(t),\dot{x}(t))=l(t,x(t),u(t))$ and $(x,u)(\cdot)\in S_{f}(t_0,x_0)$. Analogously, as in the proof of Theorem \ref{eqw-fws}, we can show that $(x,u)(\cdot)$ is an optimal pair for the value function $\mathcal{V}$ at $(t_0,x_0)$. Moreover, we have
$|l(t,x(t),u(t))|=|H^{\ast}(t,x(t),\dot{x}(t))|\leq\psi(t)$.
This implies that the selected pair $(x,u)(\cdot)$ satisfies $\tn{(B)}^+$.
\end{Rem}

\section{Stability of representations}

\noindent We denote by $I_n$ and $I$ non-degenerate closed intervals in $\R$. Let $\pi_J(\cdot)$ be a projection of $\R$ onto a nonempty closed convex subset $J$ of $\R$. We say that a sequence $\varphi_n : I_n \to \R^{\scriptscriptstyle N}$ converges locally uniformly to  $\varphi : I \to \R^{\scriptscriptstyle N}$  if $\lim_{n\to\infty} \text{d}_\mathcal{H}(I_n,I) = 0$, and  $\varphi_n \circ \pi_{I_n} : \R \to \R^{\scriptscriptstyle N}$ converges to $\varphi \circ \pi_I : \R \to \R^{\scriptscriptstyle N}$ uniformly on every compact subset of $\R$.

\begin{Lem}\label{wcsh}
Let $H_n,H:[0,\infty)\times\R^{\scriptscriptstyle N}\times\R^{\scriptscriptstyle N}\to\R$ be  measurable with respect to the first variable, continuous with respect to the second variable, and convex with respect to the third variable. Moreover, let $H_n(t,\cdot,\cdot)$ converge uniformly on compacts to $H(t,\cdot,\cdot)$ for all $t\in[0,\infty)$. Assume also that there exist $\phi(\cdot)\in L^1([0,\infty);\R)$ and $c(\cdot)\in L^1_{\mathrm{loc}}\!\big([0,\infty);[0,\infty)\big)$ such that, for all $t\in[0,\infty)$, $x,p,q\in\R^{\scriptscriptstyle N}$, $ n\in\N$, one has $H_n(t,x,0)\leq-\phi(t)$ and
$$|H_n(t,x,p)-H_n(t,x,q)|\leq c(t)(1+|x|)|p-q|.$$
Let $s_n,s_0,\eta\in[0,\infty)$ and $z_n,z_0\in\R^{\scriptscriptstyle N}$ with $(s_{n},z_n)\to(s_0,z_0)$, and let $x_n:[s_n,\infty)\to\R^{\scriptscriptstyle N}$\linebreak  be a sequence of locally absolutely continuous functions with $x_n(s_n)=z_n$ such that
\begin{equation*}
\int_{s_n}^{\infty}H_n^{\ast}(t,x_n(t),\dot{x}_n(t))\,dt\;\leq\;\eta,\;\forall\,n\in\N.
\end{equation*}
Then there exists a subsequence $\{x_{n_i}(\cdot)\}_{i\in\N}$ of the sequence $\{x_n(\cdot)\}_{n\in\N}$ that converges locally uniformly to some locally absolutely continuous function $x:[s_0,\infty)\to\R^{\scriptscriptstyle N}$ with $x(s_0)=z_0$. Moreover, the following inequality is satisfied
\begin{equation*}
\liminf_{n\to\infty}\int_{s_n}^{\infty}H_n^{\ast}(t,x_n(t),\dot{x}_n(t))\,dt\;\geq\;\int_{s_0}^{\infty}H^{\ast}(t,x(t),\dot{x}(t))\,dt.
\end{equation*}
\end{Lem}

\noindent When the interval is bounded, the above lemma can be inferred from the proofs in\linebreak \cite[Lem. 8.10]{AMaX2} and \cite[Thm. 6.5]{AM1}. Using this fact, we will show that the above lemma also holds true on an unbounded interval.

\begin{proof}[Proof of Lemma \ref{wcsh}]
Since $H_n(t,x,0)\leq-\phi(t)$  and $H_n(t,x,0)\to H(t,x,0)$ for all $t\in[0,\infty)$, $x\in\R^{\scriptscriptstyle N}$, we obtain $H(t,x,0)\leq-\phi(t)$ for all $t\in[0,\infty)$, $x\in\R^{\scriptscriptstyle N}$. Along with \eqref{tlf-1}, this implies that $H_n^{\ast}(t,x,v)\geq -H_n(t,x,0)\geq\phi(t)$ and $H^{\ast}(t,x,v)\geq -H(t,x,0)\geq\phi(t)$ for all $t\in[0,\infty)$, $x,v\in\R^{\scriptscriptstyle N}$. For locally absolutely continuous functions $y_n(\cdot)$ on $[s_n,\infty)$ and $y(\cdot)$ on $[s_0,\infty)$,
$$\Delta_n(y_n):=\int_{s_n}^{\infty}H_n^{\ast}(t,y_n(t),\dot{y}_n(t))\,dt,\qquad \Delta(y):=\int_{s_0}^{\infty}H^{\ast}(t,y(t),\dot{y}(t))\,dt.$$
Let us fix $\Delta:=\liminf_{n\to\infty}\Delta_n(x_n)$. Because of the definition
$\Delta$, there exists a subsequence (which we do not relabel) such that $\lim_{n\to\infty}\Delta_n(x_n)=\Delta$. Without loss of generality we can assume that $s_n,s_0\in[0,1)$ for all $n\in\N$. For absolutely continuous functions $y_n^k(\cdot)$ on $[s_n,k]$ and $y^k(\cdot)$ on $[s_0,k]$, where $k\in\N$, we define
$$\Delta_n^k(y_n^k):=\int_{s_n}^{k}H_n^{\ast}(t,y_n^k(t),\dot{y}_n^k(t))\,dt,\qquad \Delta^k(y^k):=\int_{s_0}^{k}H^{\ast}(t,y^k(t),\dot{y}^k(t))\,dt.$$
We denote by $x_n^k(\cdot)$ the functions $x_n(\cdot)$ restricted to $[s_n,k]$.

Let $k=1$. Since our lemma holds true over a bounded interval, there exists an increasing sequence of natural numbers $\{n{\scriptstyle[1,i]}\}_{i\in\N}$ such that  $x_{n{\scriptscriptstyle[1,i]}}^1(\cdot)$ converges locally uniformly to some absolutely continuous function $x^1(\cdot)$ defined on $[s_0,1]$ with $x^1(s_0)=z_0$ such that
$\lim_{i\to\infty}\Delta^1(x_{n{\scriptscriptstyle[1,i]}}^1)\geq\Delta^1(x^1)$.
Let $k=2$. Since our lemma holds true over a bounded interval, there exists  a subsequence  $\{n{\scriptstyle[2,i]}\}_{i\in\N}$ of the sequence $\{n{\scriptstyle[1,i]}\}_{i\in\N}$ such that  $x_{n{\scriptscriptstyle[2,i]}}^2(\cdot)$ converges locally uniformly to some absolutely continuous function $x^2(\cdot)$ defined on $[s_0,2]$ with $x^2(s_0)=z_0$ such that
$\lim_{i\to\infty}\Delta^2(x_{n{\scriptscriptstyle[2,i]}}^2)\geq\Delta^2(x^2)$.
We observe that the function $x^2(\cdot)$ is an extension of the function $x^1(\cdot)$. Moreover,  $\lim_{i\to\infty}\Delta^1(x_{n{\scriptscriptstyle[2,i]}}^1)=\lim_{i\to\infty}\Delta^1(x_{n{\scriptscriptstyle[1,i]}}^1)\geq\Delta^1(x^1)$.

By induction, for each $k\in\N$, we get a sequence $\{n{\scriptstyle[k,i]}\}_{i\in\N}$ and an absolutely continuous\linebreak function $x^k(\cdot)$ defined on $[s_0,k]$ with $x^k(s_0)=z_0$. The functions $x_{n{\scriptscriptstyle[k,i]}}^k(\cdot)$ converge locally\linebreak uniformly to  $x^k(\cdot)$ as $i\to\infty$ for all $k\in\N$. Additionally, for all $k\geq j$ and $k,j\in\N$, we have
$\{n{\scriptstyle[k,i]}\}_{i\in\N}$ is a subsequence of
$\{n{\scriptstyle[j,i]}\}_{i\in\N}$, and $x^{k}(\cdot)$ is an extension of $x^j(\cdot)$, moreover $$\lim_{i\to\infty}\Delta^j(x_{n{\scriptscriptstyle[k,i]}}^j)\;=\;\lim_{i\to\infty}\Delta^j(x_{n{\scriptscriptstyle[j,i]}}^j)\;\geq\;\Delta^j(x^j).$$

We define the function $x(\cdot)$, locally absolutely continuous on  $[s_0,\infty)$ with $x(s_0)=z_0$, by setting $x(t):=x^j(t)$ for each $t\in[s_0,j]$ and for all $j\in\N$. We observe that $\{n{\scriptstyle[k,k]}\}_{k\in\N}$ is an increasing sequence of natural numbers. Moreover, $\{n{\scriptstyle[k,k]}\}_{k=j}^{\infty}$ is a subsequence of
$\{n{\scriptstyle[j,i]}\}_{i\in\N}$ for all $j\in\N$. Additionally, the functions $x_{n{\scriptscriptstyle[k,k]}}(\cdot)$ converge locally uniformly to  $x(\cdot)$ as $k\to\infty$. We observe that, for all $j\in\N$,
$$\lim_{k\to\infty}\Delta^j(x_{n{\scriptscriptstyle[k,k]}}^j)\;=\;\lim_{i\to\infty}\Delta^j(x_{n{\scriptscriptstyle[j,i]}}^j)\;\geq\;\Delta^j(x^j).$$
In consequence, for all $j\in\N$, we have
\begin{eqnarray}\label{wcsh-1}
\Delta &=& \liminf_{k\to\infty}\Delta_{n{\scriptscriptstyle[k,k]}}(x_{n{\scriptscriptstyle[k,k]}})\;\;\geq \;\; \liminf_{k\to\infty}\Delta^j(x_{n{\scriptscriptstyle[k,k]}}^j)+\int_j^{\infty}\phi(t)\,dt\nonumber\\
&\geq & \Delta^j(x^j)+\int_j^{\infty}\phi(t)\,dt\;\;=\;\;\int_{s_0}^jH^{\ast}(t,x(t),\dot{x}(t))\,dt+\int_j^{\infty}\phi(t)\,dt.
\end{eqnarray}
By passing to the limit in \eqref{wcsh-1} as $j\to\infty$, we obtain
$\liminf_{n\to\infty}\Delta_n(x_n)\geq\Delta(x)$.
\end{proof}

\begin{Prop}\label{glscfv}
Let $H_n,H:[0,\infty)\times\R^{\scriptscriptstyle N}\times\R^{\scriptscriptstyle N}\to\R$
satisfy the assumptions of Lemma \ref{wcsh}. Assume that $A$ is a nonempty closed subset of $\R^{\scriptscriptstyle N}$. Moreover, let $V_n$ and $V$ be the value functions associated with $H_n^{\ast}$ and $H^{\ast}$, respectively.  Then, for all $(t_0,x_0)\in[0,\infty)\times A$ and any sequence $(t_{\zn},x_{\zn})\rightarrow (t_0,x_0)$ with $(t_{\zn},x_{\zn})\in[0,\infty)\times A$, we have
\begin{equation}\label{glscfv-1}
\liminf\nolimits_{\,n\,\to\,\infty}V_n(t_{\zn},x_{\zn})\geq V(t_0,x_0).
\end{equation}
\end{Prop}
\begin{proof}
In view of Theorem \ref{eqw-fws} the value functions $V_n$ and $V$ are well-defined.
Let us fix  $(t_{\zn},x_{\zn})\rightarrow (t_0,x_0)$. We denote by LS\eqref{glscfv}  the left-hand side of the inequality \eqref{glscfv}. If $\tn{LS}\eqref{glscfv}=+\infty$, then $\tn{LS}\eqref{glscfv}\geq V(t_0,x_0)$. Suppose next that $\tn{LS}\eqref{glscfv}<+\infty$. The latter, together with the definition of limit inferior, implies that there exist a real number $\eta>0$ and a subsequence (which we do not relabel) such that $\lim_{n\to\infty}V_n(t_{\zn},x_{\zn})=\tn{LS}\eqref{glscfv}$ and $V_n(t_{\zn},x_{\zn})\leq\eta$ for all $n\in\N$. By the proof of  Proposition \ref{pfv-cvp}, it follows that there exists an optimal trajectory $x_n(\cdot)$ of $V_n$ at $(t_{\zn},x_{\zn})$ for all $n\in\N$. Therefore, we obtain
\begin{equation}\label{glscfv-2}
\int_{t_{\mzn}}^{\infty}H_n^{\ast}(t,x_n(t),\dot{x}_n(t))\,dt\;=\;V_n(t_{\zn},x_{\zn})\;\leq\;\eta,\;\;\forall\,n\in\N.
\end{equation}
By Lemma \ref{wcsh} there exists a subsequence (denoted again by $\{x_n(\cdot)\}$)
converges locally uniformly to some locally absolutely continuous function $x:[t_0,\infty)\to\R^{\scriptscriptstyle N}$ with $x(t_0)=x_0$. Moreover, the following inequalities are satisfied
\begin{equation}\label{glscfv-3}
\int_{t_0}^{\infty}H^{\ast}(t,x(t),\dot{x}(t))\,dt\;\leq\; \liminf_{n\to\infty}\int_{t_{\mzn}}^{\infty}H_n^{\ast}(t,x_n(t),\dot{x}_n(t))\,dt\;<\;+\infty.
\end{equation}
Therefore, the function $t\to H^{\ast}(t,x(t),\dot{x}(t))$ must take real values for almost all $t\in[t_0,\infty)$. Hence, $\dot{x}(t)\in\D H^{\ast}(t,x(t),\cdot)$ for almost all $t\in[t_0,\infty)$. Since $A$ is closed, $x_n([t_{\zn},\infty))\subset A$, $x_n(\cdot)$ converge locally uniformly to $x(\cdot)$, we obtain $x([t_0,\infty))\subset A$. Thus, $x(\cdot)\in S_{\!\!H}(t_0,x_0)$.
The latter, together with the definition of the value function,  implies that
\begin{equation}\label{glscfv-4}
V(t_0,x_0)\;\leq\;\int_{t_0}^{\infty}H^{\ast}(t,x(t),\dot{x}(t))\,dt.
\end{equation}
Combining \eqref{glscfv-2}, \eqref{glscfv-3} and \eqref{glscfv-4} we obtain $\tn{LS}\eqref{glscfv}\geq V(t_0,x_0)$.
\end{proof}

\begin{Prop}\label{vic-prop-H}
Assume that the condition $\tn{(h)}''_H$ holds. Additionally, we assume that the Hamiltonian $H$ satisfies the following condition.
\begin{enumerate}[leftmargin=33mm]
\item[$\pmb{\tn{(VIC)}_H:}$] There exists a set $C\subset(0,\infty)$ of full measure such that
\begin{align}
\D H^{\ast}(t,x,\cdot)\;\cap\; T_A(x)\neq\emptyset\;\;\;&\tn{for all}\;\;\;(t,x)\in C\times A,\label{vc-f}\\
 -\,\D H^{\ast}(t,x,\cdot)\;\subset\; T_A(x)\;\;\;&\tn{for all}\;\;\;(t,x)\in C\times A.\label{ic-b}
\end{align}
\end{enumerate}
Then for any $(t_0,x_0)\in[0,\infty)\times A$ there exists  a locally absolutely continuous function $x(\cdot)$ belongs to $S_{\!\!H}(t_0,x_0)$. Moreover, for any $s_0,t_0\in[0,\infty)$  with $s_0<t_0$ and $x_0\in A$, and for any absolutely continuous function $x(\cdot)$ defined on $[s_0,t_0]$ such that $\dot{x}(t)\in\D H^{\ast}(t,x(t),\cdot)$\linebreak for a.e. $t\in[s_0,t_0]$ with $x(t_0)=x_0$, one has $x([s_0,t_0])\subset A$.
\end{Prop}
\begin{proof}
Let $F(t,x)=\D H^{\ast}(t,x,\cdot)$ and $P(\cdot)\equiv A$. The latter, together with  \eqref{vc-f}, implies
$$\forall\,t\in C,\;\;\;\forall\,x\in P(t),\;\;\; [\{1\}\times F(t,x)]\cap T_{\G P\,}(t,x)\neq\emptyset.$$
From \cite[Thm. 4.7]{F-P-Rz}, for any $t_0,s_0\in[0,\infty)$  with $t_0<s_0$ and $x_0\in A$, there exists an absolutely continuous function $x(\cdot)$ defined on $[t_0,s_0]$ such that $\dot{x}(t)\in F(t,x(t))$ for a.e. $t\in[t_0,s_0]$ with $x(t_0)=x_0$ and $x([t_0,s_0])\subset A$. Hence, it is not difficult to deduce that for any  $t_0\in[0,\infty)$, $x_0\in A$  there exists a locally absolutely continuous function $x(\cdot)$ defined on $[t_0,\infty)$ such that $\dot{x}(t)\in F(t,x(t))$ for a.e. $t\in[t_0,\infty)$ with $x(t_0)=x_0$ and $x([t_0,\infty))\subset A$.

 By \eqref{ic-b} we deduce that
$$\forall\,t\in C,\;\;\;\forall\,x\in P(t),\;\;\; [\{-1\}\times (-F(t,x))]\subset T_{\G P\,}(t,x).$$
From \cite[Thm. 4.10]{F-P-Rz}, for any $s_0,t_0\in[0,\infty)$  with $s_0<t_0$ and $x_0\in A$, and for any absolutely continuous function $x(\cdot)$ defined on $[s_0,t_0]$ such that $\dot{x}(t)\in F(t,x(t))$ for a.e. $t\in[s_0,t_0]$ with $x(t_0)=x_0$, one has $x([s_0,t_0])\subset A$.
\end{proof}

\begin{Prop}\label{vic-prop-fl}
Assume that the triple $(U,f,l)$ satisfies $\tn{(h)}''$. Additionally, we assume that the functions $f$ and $U$ satisfy the following condition.
\begin{enumerate}[leftmargin=33mm]
\item[$\pmb{\tn{(VIC)}:}$] There exists a set $C\subset(0,\infty)$ of full measure such that \vspace{-2mm}
\begin{align*}
 f(t,x,U(t))\;\cap\; T_A(x)\neq\emptyset\;\;\;&\tn{for all}\;\;\;(t,x)\in C\times A,
 \\
 -f(t,x,U(t))\;\subset\; T_A(x)\;\;\;&\tn{for all}\;\;\;(t,x)\in C\times A.
\end{align*}
\end{enumerate}

\vspace{-4.5mm}
\pagebreak
\noindent Then for any  $t_0\in[0,\infty)$, $x_0\in A$  there exists a locally absolutely continuous function $x(\cdot)$ and a measurable function $u(\cdot)$  such that $(x,u)(\cdot)\in S_{\!\!f\,}(t_0,x_0)$. Moreover,  for any\linebreak $s_0,t_0\in[0,\infty)$  with $s_0<t_0$ and $x_0\in A$, and for any measurable function $u(\cdot)$  with $u(t)\in U(t)$ a.e. $t\in[s_0,t_0]$, there exists exactly one absolutely continuous function $x(\cdot)$ defined on $[s_0,t_0]$ such that $\dot{x}(t)=f(t,x(t),u(t))$ for a.e. $t\in[s_0,t_0]$ with $x(t_0)=x_0$ and $x([s_0,t_0])\subset A$.
\end{Prop}
\begin{proof}
In view of  Theorem \ref{rwwar-i}, we obtain that the assertions of Proposition \ref{vic-prop-H} are\linebreak satisfied. The latter, together with \cite[Thm. 8.2.10]{A-F}, implies that for any  $t_0\in[0,\infty)$, $x_0\in A$\linebreak  there exists a locally absolutely continuous function $x(\cdot)$ and a measurable function $u(\cdot)$  such that $(x,u)(\cdot)\in S_{\!\!f\,}(t_0,x_0)$. Moreover,  for any $s_0,t_0\in[0,\infty)$  with $s_0<t_0$ and $x_0\in A$, and for any measurable function $u(\cdot)$  with $u(t)\in U(t)$ a.e. $t\in[s_0,t_0]$, there exists exactly one absolutely continuous function $x(\cdot)$ defined on $[s_0,t_0]$ such that $\dot{x}(t)=f(t,x(t),u(t))$ for a.e. $t\in[s_0,t_0]$ with $x(t_0)=x_0$ and $x([s_0,t_0])\subset A$.
\end{proof}

\noindent\bf{Assumption (LB).} For any $r>0$ there exists  $\psi_r\in L^1([0,\infty);[0,\infty))$ such that for all $(t_0,x_0)\in[0,\infty)\times (A\cap r\B)$ and any $(x,u)(\cdot)\in S_{\!\!f\,}(t_0,x_0)$  one has $|l(t,x(t),u(t))|\leq\psi_r(t)$ for almost every $t\in[t_0,\infty)$.

\begin{Rem}\label{sb-rem}
If the triple $(U,f,l)$ satisfies the conditions $\tn{(h)}''$, (VIC), and (LB), then the value function $\mathcal{V}$ is real-valued on $[0,\infty)\times A$. If
$\D \mathcal{V}=[0,\infty)\times A$, then (LB) is a local version of (B). Observe, however, that (LB) by itself does not imply that $\mathcal{V}$ vanishes at infinity; we will nonetheless need it to establish convergence.
\end{Rem}

\begin{Prop}\label{guscfv}
Assume that $(U,f_n,l_n)$ and $(U,f,l)$ satisfy  $\tn{(h)}''$ and \tn{(LB)} with the same functions $\phi$, $c$, $\{\psi_r\}$ and the set $A$. Let $f_n$ and $f$ satisfy \tn{(VIC)}. Moreover, assume that $f_n(t,\cdot,\cdot)$ converge to $f(t,\cdot,\cdot)$ and $l_n(t,\cdot,\cdot)$ converge to $l(t,\cdot,\cdot)$ uniformly on compacts in $\R^{\scriptscriptstyle N}\times\R^{\scriptscriptstyle M}$ for every $t\in[0,\infty)$.
Let $\mathcal{V}_n$ and $\mathcal{V}$ be the value functions associated with $(U,f_n,l_n)$ and $(U,f,l)$, respectively. Then, for all $(t_0,x_0)\in[0,\infty)\times A$ there exist a sequence $x_{0i}\rightarrow x_0$ with  $x_{0i}\in A$ and an increasing sequence of natural numbers $\{n_i\}_{i\in\N}$ such that
\begin{equation}\label{guscfv-1}
\limsup\nolimits_{\,i\,\to\,\infty}\mathcal{V}_{n_i}(t_{0},x_{0i})\leq \mathcal{V}(t_0,x_0).
\end{equation}

\end{Prop}
\begin{proof}
By Remark \ref{sb-rem} the value functions $\mathcal{V}_n$ and $\mathcal{V}$ are a real function on $[0,\infty)\times A$.\linebreak Fix $(t_0,x_0)\in[0,\infty)\times A$. In view of Proposition \ref{pfv-cvp} and Theorem \ref{eqw-fws} there exists an optimal pair $(x,u)(\cdot)$  of $\mathcal{V}$ at  $(t_0,x_0)$, i.e. $(x,u)(\cdot)\in S_{\!\!f\,}(t_0,x_0)$ and
\begin{equation}\label{guscfv-2}
\mathcal{V}(t_0,x_0)=\int_{t_0}^{\infty}l(t,x(t),u(t))\,dt
\end{equation}
Since $x([t_0,\infty))\subset A$, it follows that $x(i)\in A$ for all $i>t_0$ and $i\in\N$. By Proposition \ref{vic-prop-fl}  there exists a locally absolutely continuous function $\tilde{x}_n^i(\cdot)$ and a measurable function $\tilde{u}_n^i(\cdot)$  such that $(\tilde{x}_n^i,\tilde{u}_n^i)(\cdot)\in S_{\!\!f_n}(i,x(i))$. Moreover, there exists absolutely continuous function $\bar{x}_n^i(\cdot)$ defined on $[t_0,i]$ such that $\dot{\bar{x}}_n^i(t)=f_n(t,\bar{x}_n^i(t),u(t))$ for a.e. $t\in[t_0,i]$ with $\bar{x}_n^i(i)=x(i)$ and $\bar{x}_n^i([t_0,i])\subset A$. We define the following functions on the set $[t_0,\infty)$ by the formula
$$x_n^i(t):=
\left\{
\begin{array}{ccl}
\bar{x}_n^i(t) & \tn{if} & t\in[t_0,i],\\[1mm]
\tilde{x}_n^i(t) & \tn{if} & t\in[i,\infty),
\end{array}
\right.
\qquad
u_n^i(t):=
\left\{
\begin{array}{lcl}
u(t) & \tn{if} & t\in[t_0,i],\\[1mm]
\tilde{u}_n^i(t) & \tn{if} & t\in[i,\infty).
\end{array}
\right.$$
We observe that $(x_n^i,u_n^i)(\cdot)\in S_{\!\!f_n}(t_0,x_n^i(t_0))$, $x_n^i(i)=x(i)$ and $u_n^i(\cdot)=u(\cdot)$ on $[t_0,i]$ for all $i>t_0$, $i,n\in\N$. The latter, together with (h6), implies that
\begin{align*}
&|x_n^i(t)-x(t)| \;\leq\; \int_t^i|\dot{x}_n^i(s)-\dot{x}(s)|\,d\!s\;=\;\int_t^i|f_n(s,x_n^i(s),u(s))-f(s,x(s),u(s))|\,d\!s\\
&\leq\;\int_t^i|f_n(s,x_n^i(s),u(s))-f(s,x_n^i(s),u(s))|\,d\!s+\int_t^i|f(s,x_n^i(s),u(s))-f(s,x(s),u(s))|\,d\!s
\\
&\leq\;\int_{t_0}^i|f_n(s,x_n^i(s),u(s))-f(s,x_n^i(s),u(s))|\,d\!s+\int_t^ik(s)\,|x_n^i(s)-x(s)|\,d\!s
\end{align*}
for all $t\in[t_0,i]$, $i>t_0$, $i,n\in\N$, where $k(s)$ is the Lipschitz constant of the function $f(s,\cdot,u(s))$. By  Gronwall’s lemma, we conclude that
\begin{align}
&\sup\nolimits_{t\in[t_0,i]}|x_n^i(t)-x(t)|\;\leq\;\Delta_n^i\,\exp\big(\textstyle\int_{t_0}^ik(s)\,d\!s\big), \label{guscfv-3}\\[1mm]
& \Delta_n^i:=\int_{t_0}^i|f_n(s,x_n^i(s),u(s))-f(s,x_n^i(s),u(s))|\,d\!s, \label{guscfv-4}
\end{align}
for all $i>t_0$, $i,n\in\N$. By (h2) we get $|\dot{x}_n^i(t)|\leq c(t)(1+|x_n^i(t)|)$ for a.e. $t\in[t_0,i]$. The latter,\linebreak together with Gronwall’s lemma, implies that $|x_n^i(t)|\leq (1+|x(i)|)\exp(2\int_{t_0}^ic(s)\,d\!s)=:r_i$\linebreak for all $t\in[t_0,i]$.
Hence, $x_n^i(s)\in r_i\B$ for all $s\in[t_0,i]$, $i>t_0$, $i,n\in\N$. By \eqref{guscfv-4} we obtain\linebreak $\Delta_n^i\leq\int_{t_0}^i\Gamma_n^i(s)\,d\!s$ for all $i>t_0$, $i,n\in\N$, where $\Gamma_n^i(s):=\sup_{x\,\in\, r_i\B}|f_n(s,x,u(s))-f(s,x,u(s))|$. Since $|f_n(s,x,u(s))|\leq c(s)(1+r_i)$ and $|f(s,x,u(s))|\leq c(s)(1+r_i)$ for all $x\in r_i\B$, we have $\Gamma_n^i(s)\leq 2c(s)(1+r_i)$ for all  $s\in[t_0,i]$, $i>t_0$, $i,n\in\N$.
Since $f_n(s,\cdot,u(s))$ converge uniformly on compacts to $f(s,\cdot,u(s))$ for all  $s\in[t_0,i]$, we get $\lim_{n\to\infty}\Gamma_n^i(s)=0$ for all  $s\in[t_0,i]$, $i>t_0$.\linebreak Thus, by Lebesgue's dominated convergence theorem, we obtain $\lim_{n\to\infty }\int_{t_0}^i\Gamma_n^i(s)\,d\!s=0$\linebreak for all $i>t_0$. The latter, together with $\Delta_n^i\leq\int_{t_0}^i\Gamma_n^i(s)\,d\!s$, implies  $\lim_{n\to\infty}\Delta_n^i=0$ for all $i>t_0$. Hence, we can choose an increasing sequence of natural numbers $\{n_i\}_{i\in\N}$ such that  $\Delta_{n_i}^i\exp(\int_{t_0}^ik(s)\,d\!s)\leq\frac{1}{i}$ holds for all $i>t_0$.
Therefore, by the inequality \eqref{guscfv-3}, we conclude
\begin{equation}\label{guscfv-5}
\textstyle |x_{n_i}^i(t_0)-x_0|\;=\;|x_{n_i}^i(t_0)-x(t_0)|\;\leq\;\frac{1}{i}\;\;\tn{for all}\;\;i>t_0,\; i\in\N.
\end{equation}

Let $x_i(t):=x_{n_i}^i(t)$ and $u_i(t):=u_{n_i}^i(t)$ for all $t\in[t_0,\infty)$, $i>t_0$, $i\in\N$. We observe that $(x_i,u_i)(\cdot)\in S_{\!\!f_{n_i}}(t_0,x_i(t_0))$, $x_i(i)=x(i)$ and $u_i(\cdot)=u(\cdot)$ on $[t_0,i]$ for all $i>t_0$, $i\in\N$. Moreover,
in view of \eqref{guscfv-5}, we have $|x_i(t_0)-x(t_0)|\leq\frac{1}{i}$ for all $i>t_0$, $i\in\N$. In particular $x_i(t_0)\in r\B$ with $r:=1+|x_0|$ for all $i>t_0$, $i\in\N$.
In view of (LB) we have
\begin{equation}\label{guscfv-6}
|l_{n_i}(t,x_i(t),u_i(t))|\leq\psi_r(t)\;\;\tn{for a.e.}\;\;[t_0,\infty),\;\; \forall\,i>t_0.
\end{equation}
Let $t_0<T<i$. Then, in view of (h6), we have
\begin{align*}
&|x_i(t)-x(t)| \;\leq\;\frac{1}{i}+\int_{t_0}^t|\dot{x}_i(s)-\dot{x}(s)|\,d\!s\;=\;\frac{1}{i}+\int_{t_0}^t|f_{n_i}(s,x_i(s),u(s))-f(s,x(s),u(s))|\,d\!s\\
&\leq\;\frac{1}{i}+\int_{t_0}^t|f_{n_i}(s,x_i(s),u(s))-f(s,x_i(s),u(s))|\,d\!s+\int_{t_0}^t|f(s,x_i(s),u(s))-f(s,x(s),u(s))|\,d\!s\\
&\leq\;\frac{1}{i}+\int_{t_0}^T|f_{n_i}(s,x_i(s),u(s))-f(s,x_i(s),u(s))|\,d\!s+\int_{t_0}^tk(s)\,|x_i(s)-x(s)|\,d\!s
\end{align*}

\vspace{-3mm}
\pagebreak

\noindent for all $t\in[t_0,T]$, where $k(s)$ is the Lipschitz constant of  $f(s,\cdot,u(s))$. By  Gronwall’s lemma,
\begin{align}
&\textstyle\sup\nolimits_{t\in[t_0,T]}|x_i(t)-x(t)|\;\leq\;\big(\frac{1}{i}+\Delta_i^T\big)\,\exp\big(\int_{t_0}^Tk(s)\,d\!s\big), \label{guscfv-7}\\[1mm]
& \Delta_i^T:=\int_{t_0}^T|f_{n_i}(s,x_i(s),u(s))-f(s,x_i(s),u(s))|\,d\!s, \label{guscfv-8}
\end{align}
In view of (h2) we have $|\dot{x}_i(t)|\leq c(t)(1+|x_i(t)|)$ for a.e. $t\in[t_0,T]$. The latter, together with Gronwall’s lemma, implies  $|x_i(t)|\leq (1+r)\exp(2\int_{t_0}^Tc(s)\,d\!s)=:R_T$ for all $t\in[t_0,T]$.
Hence,\linebreak $x_i(s)\in R_T\B$ for all $s\in[t_0,T]$. In view of \eqref{guscfv-8} we obtain that $\Delta_i^T\leq\int_{t_0}^T\Gamma_i^T(s)\,d\!s$, where $\Gamma_i^T(s):=\sup_{x\,\in\, R_T\B}|f_{n_i}(s,x,u(s))-f(s,x,u(s))|$. Since $|f_{n_i}(s,x,u(s))|\leq c(s)(1+R_T)$ and $|f(s,x,u(s))|\leq c(s)(1+R_T)$ for all $x\in R_T\B$, we have $\Gamma_i^T(s)\leq 2c(s)(1+R_T)$ for all\linebreak  $s\in[t_0,T]$.
Since $f_{n_i}(s,\cdot,u(s))$ converge uniformly on compacts to $f(s,\cdot,u(s))$ for all\linebreak $s\in[t_0,T]$, we obtain $\lim_{i\to\infty}\Gamma_i^T(s)=0$ for all $s\in[t_0,T]$. Therefore, by Lebesgue's\linebreak dominated convergence theorem, we get $\lim_{i\to\infty }\int_{t_0}^T\Gamma_i^T(s)\,d\!s=0$. The latter, together with\linebreak $\Delta_i^T\leq\int_{t_0}^T\Gamma_i^T(s)\,d\!s$, implies $\lim_{i\to\infty}\Delta_i^T=0$. So, by  \eqref{guscfv-7}, $\lim x_i(s)=x(s)$ for all $s\in[t_0,T]$.\linebreak
Since $l_{n_i}(s,\cdot,u(s))$ converge uniformly on compacts to $l(s,\cdot,u(s))$ for all $s\in[t_0,T]$,\linebreak we obtain that $\lim_{i\to\infty}l_{n_i}(s,x_i(s),u(s))=l(s,x(s),u(s))$ for all $s\in[t_0,T]$. Additionally,\linebreak $|l_{n_i}(s,x_i(s),u(s))|\leq c(s)(1+R_T)$ for all $s\in[t_0,T]$. Therefore, by Lebesgue's dominated convergence theorem, we have
\begin{equation}\label{guscfv-9}
\lim_{i\to\infty}\int_{t_0}^Tl_{n_i}(s,x_i(s),u(s))\,d\!s=\int_{t_0}^Tl(s,x(s),u(s))\,d\!s.
\end{equation}
Combining \eqref{guscfv-6}, \eqref{guscfv-9} and $(x_i,u_i)(\cdot)\in S_{\!\!f_{n_i}}(t_0,x_i(t_0))$ we obtain
\begin{eqnarray}\label{guscfv-10}
\limsup_{\,i\,\to\,\infty}\mathcal{V}_{n_i}(t_{0},x_i(t_0)) &\leq & \limsup_{\,i\,\to\,\infty}\int_{t_0}^\infty l_{n_i}(s,x_i(s),u_i(s))\,d\!s\nonumber\\
&\leq & \lim_{i\to\infty}\int_{t_0}^Tl_{n_i}(s,x_i(s),u(s))\,d\!s+ \int_{T}^\infty \psi_r(t)\,d\!s\nonumber\\
&=& \int_{t_0}^Tl(s,x(s),u(s))\,d\!s+ \int_{T}^\infty \psi_r(t)\,d\!s.
\end{eqnarray}
Passing to the limit in \eqref{guscfv-10} as $T\to\infty$ and using \eqref{guscfv-2}, we get the following inequality $\limsup_{\,i\,\to\,\infty}\mathcal{V}_{n_i}(t_{0},x_i(t_0))\leq \mathcal{V}(t_0,x_0)$.
By setting $x_{0i}:=x_i(t_0)$, we obtain  \eqref{guscfv-1}.
\end{proof}

\begin{Def}\label{sch-def}
We say that  $H$ belongs to $\mathscr{H}(\theta,\lambda,\{\psi_r\},\phi,c,A)$ if $\textit{I\!H}$ given by the formula
\begin{equation}\label{sch-def-1}
\textit{I\!H}(t,x,p):=[\theta(t)]^{-1}H(t,x,\theta(t)p)
\end{equation}
satisfies $\tn{(h)}''_{\textit{I\!H}}$ with the functions $\lambda$, $\phi$, $c$, and the set $A$.  Additionally, we require that
\begin{equation}\label{sch-def-2}
\theta(t)\lambda(t,x(t))\leq\psi_r(t)\;\tn{for a.e.}\; t\in[t_0,\infty),
\end{equation}
for any $x(\cdot)\in S_{\!\!\textit{I\!H}}(t_0,x_0)$, and for all $(t_0,x_0)\in[0,\infty)\times (A\cap r\B)$, and for every $r>0$, where $\theta\in L^\infty([0,\infty);\R^+)$ and $\psi_r\in L^1([0,\infty);[0,\infty))$.
\end{Def}

\vspace{-4.5mm}
\pagebreak

\begin{Rem}\label{sch-rem}
If $A$  is a non-degenerate compact interval in  $\R$, then the Hamiltonian $H$ from Example \ref{Ex1} satisfies $\tn{(OPC)}_H$ and belongs to  $\mathscr{H}(e^{-\gamma t},\lambda,\{\psi_r\},\phi,c,A)$ with
\begin{align*}
&\lambda(t,x):=\alpha(t)|x|+\alpha(t)+1,\;\;\;\;\;c(t):=\alpha(t)+1,\\
& \psi_r(t):=e^{-\gamma t}\alpha(t)\|A\|+e^{-\gamma t}\alpha(t)+e^{-\gamma t},\;\,\phi\equiv 0.
\end{align*}

If $A$ is an unbounded closed interval in $\R$, then the Hamiltonian $H$ from Example \ref{Ex1} satisfies $\tn{(OPC)}_H$ and belongs to  $\mathscr{H}(e^{-\gamma t},\lambda,\{\psi_r\},\phi,c,A)$ with $\gamma>\|\alpha\|_{\infty}$ and
\begin{align*}
&\lambda(t,x):=\alpha(t)|x|+\alpha(t)+1,\;\,c(t):=\alpha(t)+1,\;\;\phi\equiv 0,\\
& \psi_r(t):=(r+t)\,\alpha(t)\,e^{(\|\alpha\|_{\infty}-\gamma)\,t}+\,\alpha(t)\,e^{-\gamma t}+\,e^{-\gamma t}.
\end{align*}
We show that \eqref{sch-def-2} holds in this case. Let $x(\cdot)\in S_{\!\!\textit{I\!H}}(t_0,x_0)$. Then $|\dot{x}(t)|\leq\alpha(t)|x(t)|+1$.\linebreak
By Gronwall’s lemma, we obtain
$|x(t)|\leq (|x_0|+t-t_0)\exp(\|\alpha\|_{\infty}(t-t_0))$.
The latter, together with $x_0\in r\B$, implies that $|x(t)|\leq (r+t)\exp(\|\alpha\|_{\infty}t)$. Therefore, we have
$$e^{-\gamma t}\lambda(t,x(t))\leq e^{-\gamma t}\big(\alpha(t)(r+t)e^{\|\alpha\|_{\infty}t}+\,\alpha(t)+1\big)=\psi_r(t).$$

Moreover, note that if $H\in\mathscr{H}(\theta,\lambda,\{\psi_r\},\phi,c,A)$, then $H$ satisfies  $\tn{(h)}''_H$ with $\tau\lambda$, $-\tau|\phi|$, $\tau c$, where $\tau:=1+\|\theta\|_{\infty}$.
\end{Rem}

\begin{Prop}\label{sch-prop}
Assume that the Hamiltonian $H$ belongs to $\mathscr{H}(\theta,\lambda,\{\psi_r\},\phi,c,A)$. Then there exists a representation $(\B,f,l)$ of  $H$ satisfying $\tn{(LB)}$ with $\{20\psi_r\}$ and satisfying $\tn{(h)}''$ with $U\equiv\B$, $-\tau|\phi|$,  $\tau c$, where $\tau:=40(1+N)(1+\|\theta\|_{\infty})$. Moreover, if $H$ satisfies  $\tn{(OPC)}_H$, then the triple $(\B,f,l)$ satisfies $\tn{(OPC)}$.
\end{Prop}
\begin{proof}
In view of Theorem \ref{rep-t}, there exist functions $f:[0,\infty)\times\R^{\scriptscriptstyle N}\times\R^{\scriptscriptstyle N+1}\to\R^{\scriptscriptstyle N}$ and
$\textit{l\!l}:[0,\infty)\times\R^{\scriptscriptstyle N}\times\R^{\scriptscriptstyle N+1}\to\R$,  measurable in $\,t\,$ for all $(x,u)\in\R^{\scriptscriptstyle N}\times\R^{\scriptscriptstyle N+1}$ and continuous in $(x,u)$ for all $t\in[0,\infty)$, such that the triple $(\B,f,\textit{l\!l})$ is a representation of $\textit{I\!H}$ and $f(t,x,\B)=\D\textit{I\!H}^{\ast}(t,x,\cdot)$ for all $t\in[0,\infty)$, $x\in\R^{\scriptscriptstyle N}$. Additionally, the conditions  (A1)-(A5) from\linebreak Theorem \ref{rep-t} are satisfied. By \eqref{sch-def-1}, we have $H(t,x,p)=\theta(t)\textit{I\!H}(t,x,p/\theta(t))$. Thus,
\begin{eqnarray*}
H(t,x,p)
&=& \theta(t)\sup\nolimits_{u\in \B}\,\{\,\langle\,p/\theta(t),f(t,x,u)\,\rangle-\textit{l\!l}(t,x,u)\,\}\\
&=&\sup\nolimits_{u\in \B}\,\{\,\langle\, p,f(t,x,u)\,\rangle-\theta(t)\textit{l\!l}(t,x,u)\,\}.
\end{eqnarray*}
Therefore, the triple $(\B,f,l)$ is a representation of $H$, where $l(t,x,u):=\theta(t)\textit{l\!l}(t,x,u)$. Since $H^{\ast}(t,x,v)=\theta(t)\textit{I\!H}^{\ast}(t,x,v)$, we get $\D\textit{I\!H}^{\ast}(t,x,\cdot)=\D H^{\ast}(t,x,\cdot)$. In particular, $f(t,x,\B)=\D H^{\ast}(t,x,\cdot)$ and  $S_{\!\!\textit{I\!H}}(t_0,x_0)=S_{\!\!H}(t_0,x_0)$. Hence, the condition $\tn{(OPC)}$ follows directly from the condition $\tn{(OPC)}_H$. Moreover, in view of (A2) and \eqref{sch-def-2}, we obtain
$$|l(t,x(t),u(t))|=\theta(t)|\textit{l\!l}(t,x(t),u(t))|\leq 20\,\theta(t)\lambda(t,x(t))\leq 20\,\psi_r(t)\;\tn{for a.e.}\; t\in[t_0,\infty),$$
for all $(x,u)(\cdot)\in S_{\!\!f\,}(t_0,x_0)$ and  $(t_0,x_0)\in[0,\infty)\times(A\cap r\B)$,
Therefore, the triple $(\B,f,l)$ satisfies $\tn{(LB)}$ with $\{20\psi_r\}$. Similarly as in Theorem \ref{rwwar}, it can be shown that the triple $(\B,f,l)$ satisfies $\tn{(h)}''$ with $U(\cdot)\equiv\B$, $-\tau|\phi|$, $\tau c$, where $\tau:=40(1+N)(1+\|\theta\|_{\infty})$.
\end{proof}

From \cite[Thm. 6.6]{AM}, we obtain the following theorem.

\begin{Th}\label{ccsrep-thm}
Assume that $H_n$ satisfies $\tn{(h)}_{H_n}''$ with $\lambda_n$ and $H$ satisfies $\tn{(h)}_{H}''$ with $\lambda$.\linebreak Consider the representations $(\B,f_n,l_n)$ and $(\B,f,l)$ of $H_n$ and $H$, respectively, defined as in the proof of Theorem~\ref{rep-t}. If $H_n(t,\cdot,\cdot)$ converge uniformly on compacts to $H(t,\cdot,\cdot)$\linebreak  and $\lambda_n(t,\cdot)$ converge uniformly on compacts to $\lambda(t,\cdot)$  for every $t\in[0,\infty)$, then $f_n(t,\cdot,\cdot)$\linebreak converge to $f(t,\cdot,\cdot)$ and $l_n(t,\cdot,\cdot)$ converge to $l(t,\cdot,\cdot)$ uniformly on compacts in $\R^{\scriptscriptstyle N}\times\R^{\scriptscriptstyle N+1}$  for every $t\in[0,\infty)$.
\end{Th}

From Theorem \ref{ccsrep-thm} and the proof of Proposition \ref{sch-prop}, we get the following corollary.

\begin{Cor}\label{ccsrep-cor}
Let $H_n\in\mathscr{H}(\theta,\lambda_n,\{\psi_r\},\phi,c,A)$ and $H\!\in\!\mathscr{H}(\theta,\lambda,\{\psi_r\},\phi,c,A)$. Moreover,\linebreak assume that $H_n(t,\cdot,\cdot)$ converge uniformly on compacts to $H(t,\cdot,\cdot)$ and $\lambda_n(t,\cdot)$ converge uniformly on compacts to $\lambda(t,\cdot)$  for all $t\in[0,\infty)$. Then there exist the representations $(\B,f_n,l_n)$ and $(\B,f,l)$ of $H_n$ and $H$, respectively, satisfying $\tn{(h)}''$ and $\tn{(LB)}$ with same $U\equiv\B$, $A$, $\{20\psi_r\}$, $-\tau|\phi|$, $\tau c$,  where $\tau:=40(1+N)(1+\|\theta\|_{\infty})$. Moreover, $f_n(t,\cdot,\cdot)$ converge to\linebreak $f(t,\cdot,\cdot)$ and $l_n(t,\cdot,\cdot)$ converge to $l(t,\cdot,\cdot)$ uniformly on compacts in $\R^{\scriptscriptstyle N}\!\times\R^{\scriptscriptstyle N+1}$ for all $t\!\in\![0,\!\infty)$.
\end{Cor}

From Propositions \ref{glscfv} and \ref{guscfv}, and Corollary \ref{ccsrep-cor}, we get the following theorem.

\begin{Th}\label{thmsr}
Assume that $H_n\in\mathscr{H}(\theta,\lambda_n,\{\psi_r\},\phi,c,A)$ satisfies the condition $\tn{(VIC)}_{H_n}$ and $H\in\mathscr{H}(\theta,\lambda,\{\psi_r\},\phi,c,A)$ satisfies the condition $\tn{(VIC)}_{H}$. Additionally, let $H_n(t,\cdot,\cdot)$\linebreak converge uniformly on compacts to $H(t,\cdot,\cdot)$ and $\lambda_n(t,\cdot)$ converge uniformly on compacts to $\lambda(t,\cdot)$  for all $t\in[0,\infty)$. If $V_n$ and $V$ are the value functions associated with $H_n^{\ast}$ and $H^{\ast}$, respectively, then \tn{e}-$\liminf_{n\to\infty}V_n=V$.
\end{Th}

Below we provide an example of the Hamiltonians $H_n$ and $H$ satisfying the assumptions of Theorem \ref{thmsr}. In addition, $H_n$ satisfies $\tn{(B)}_{H_n}$ and $\tn{(OPC)}_{H_n}$. Therefore, the value function $V_n$ is the unique solution of \eqref{eqhjb} with $H_n$. However, the limit Hamiltonian $H$ satisfies $\tn{(B)}_{H}$, but does not satisfy $\tn{(OPC)}_{H}$. Therefore, the limit value function $V$ may not be the unique solution of \eqref{eqhjb} with $H$.

\begin{Ex}\label{Ex2}
Let us define the Hamiltonian $H_n:[0,\infty)\times\R\times\R\rightarrow\R$ by the formula:
\begin{equation*}
\textstyle H_n(t,x,p):=\max\big\{\,\alpha(t)\,\max\{p,0\}\,|x|-\alpha(t)\,e^{-\gamma t},0\,\big\}+\frac{1}{n}\max\{p,0\}-\frac{1}{n}\,\beta(t)\,e^{-2\gamma t},
\end{equation*}
where $\alpha(t)\in L^\infty([0,\infty);\R^+)$, $\,\beta(t)\in L^\infty([0,\infty);\R)$, $\gamma\in\R^+$. Let $A=(-\infty,0]$.

\noindent We observe that
\begin{align*}
& H_n^{\ast}(t,x,v)\;=\;\frac{\beta(t)\,e^{-2\gamma t}}{n}+\left\{
\begin{array}{ccl}
+\infty, & \tn{if} & v\notin[0,\,\alpha(t)\,|x|+\!1/n],\;x\not=0,\\[1mm]
\max\left\{\frac{\displaystyle v-1/n}{\displaystyle e^{\gamma t}\,|x|},\,0\right\}\!\!, & \tn{if} & v\in[0,\,\alpha(t)\,|x|+\!1/n],\;x\not=0, \\[3mm]
0, & \tn{if} & v\in[0,1/n],\; x=0,\\
+\infty, & \tn{if} & v\notin[0,1/n],\;x=0,
\end{array}
\right.\\[2mm]
& \D H_n^{\ast}(t,x,\cdot)=[0,\,\alpha(t)\,|x|+\!1/n]\;\,\tn{for all}\;\, t\in[0,\infty),\, x\in\R.
\end{align*}
The Hamiltonian $H_n$ satisfies $\tn{(VIC)}_{H_n}$, $\tn{(B)}_{H_n}$ and $\tn{(OPC)}_{H_n}$. Moreover, $H_n(t,\cdot,\cdot)$ converge uniformly on compacts to $H(t,\cdot,\cdot)$, where $H$ is given by the formula:
\begin{equation*}
 H(t,x,p):=\max\big\{\,\alpha(t)\,\max\{p,0\}\,|x|-\alpha(t)\,e^{-\gamma t},\,0\,\big\}.
\end{equation*}
We observe that
\begin{align*}
& H^{\ast}(t,x,v)\;=\;\left\{
\begin{array}{ccl}
+\infty, & \tn{if} & v\notin[0,\,\alpha(t)\,|x|\,],\;x\not=0,\\[1mm]
\frac{\displaystyle v}{\displaystyle e^{\gamma t}\,|x|}, & \tn{if} & v\in[0,\,\alpha(t)\,|x|\,],\;x\not=0, \\[3mm]
0, & \tn{if} & v=0,\; x=0,\\
+\infty, & \tn{if} & v\neq 0,\;x=0,
\end{array}
\right.\\[2mm]
& \D H^{\ast}(t,x,\cdot)=[0,\,\alpha(t)\,|x|\,]\;\,\tn{for all}\;\, t\in[0,\infty),\, x\in\R.
\end{align*}
The Hamiltonian $H$ satisfies $\tn{(VIC)}_{H}$ and $\tn{(B)}_{H}$, but does not satisfy $\tn{(OPC)}_{H}$. Moreover, $H_n, H\in\mathscr{H}(e^{-\gamma t},\lambda,\{\psi_r\},\phi,c,A)$ with
\begin{align*}
&\lambda(t,x):=\alpha(t)|x|+\alpha(t)+|\,\beta(t)|+1,\;\;\;\;\;c(t):=\alpha(t)+|\,\beta(t)|+1,\\
& \psi_r(t):=e^{-\gamma t}\alpha(t)\,r+e^{-\gamma t}\alpha(t)+e^{-\gamma t}\,|\,\beta(t)|+e^{-\gamma t},\;\;\phi(t):=-e^{-\gamma t}\,|\,\beta(t)|.
\end{align*}
We show that \eqref{sch-def-2} holds in this case. Let $x(\cdot)\in S_{\!\!\textit{I\!H}}(t_0,x_0)$ and $y(\cdot)\in S_{\!\!\textit{I\!H}_n}(t_0,x_0)$. Then $\dot{x}(t)\geq 0$ and $\dot{y}(t)\geq 0$ for a.e. $t\in[t_0,\infty)$. Thus, the functions $x(\cdot)$ and $y(\cdot)$ are non-decreasing. So, for all $x_0\in A$ we get $x_0\leq x(t)\leq 0$ and $x_0\leq y(t)\leq 0$ for all $t\in[t_0,\infty)$. Therefore, for all $x_0\in A\cap r\B$ we have $|x(t)|\leq r$ and $|y(t)|\leq r$ for all $t\in[t_0,\infty)$. Hence,
\begin{align*}
& e^{-\gamma t}\lambda(t,x(t))=e^{-\gamma t}\big(\alpha(t)|x(t)|+\alpha(t)+|\,\beta(t)|+1\big)\leq\psi_r(t),\\
& e^{-\gamma t}\lambda(t,y(t))=e^{-\gamma t}\big(\alpha(t)|y(t)|+\alpha(t)+|\,\beta(t)|+1\big)\leq\psi_r(t).
\end{align*}
\end{Ex}


\end{document}